\documentclass[11pt]{amsart}
\usepackage{amssymb,amsmath,epsfig,mathrsfs, enumerate, xparse, mathtools}
\usepackage[pagewise]{lineno}
\usepackage[nodisplayskipstretch]{setspace}
\usepackage{graphicx}\usepackage[normalem]{ulem}
\usepackage{fancyhdr}
\usepackage[margin=2.5cm]{geometry}
\usepackage{hyperref}
\hypersetup{
 colorlinks   = true,
 urlcolor     = blue,
 linkcolor    = blue,
 citecolor   = red ,
 bookmarksopen=true
}

\theoremstyle{plain}
\newtheorem{thrm}{Theorem}[section]
\newtheorem{definition}{Definition}
\newtheorem{lemma}[thrm]{Lemma}
\newtheorem{prop}[thrm]{Proposition}

\newtheorem{rmrk}[thrm]{Remark}

\allowdisplaybreaks
\begin{document}
\newcommand{\Q}{\mathbb Q}
\newcommand{\B}{\mathbb B}
\newcommand{\sn}{\mathbb{S}^{n-1}}
\newcommand{\SL}{\mathcal L^{1,p}( D)}
\newcommand{\Lp}{L^p( Dega)}
\newcommand{\La}{\mathscr{L}_a}
\newcommand{\CO}{C^\infty_0( \Omega)}
\newcommand{\Rn}{\mathbb R^n}
\newcommand{\Rm}{\mathbb R^m}
\newcommand{\R}{\mathbb R}
\newcommand{\Om}{\Omega}
\newcommand{\Hn}{\mathbb H^n}
\newcommand{\aB}{\alpha B}
\newcommand{\eps}{\ve}
\newcommand{\BVX}{BV_X(\Omega)}
\newcommand{\p}{\partial}
\newcommand{\IO}{\int_\Omega}
\newcommand{\bG}{\boldsymbol{G}}
\newcommand{\bg}{\mathfrak g}
\newcommand{\bz}{\mathfrak z}
\newcommand{\bv}{\mathfrak v}
\newcommand{\Bux}{\mbox{Box}}
\newcommand{\e}{\ve}
\newcommand{\X}{\mathcal X}
\newcommand{\Y}{\mathcal Y}
\newcommand{\W}{\mathcal W}
\newcommand{\la}{\lambda}
\newcommand{\vf}{\varphi}
\newcommand{\rhh}{|\nabla_H \rho|}
\newcommand{\Ba}{\mathcal{B}_\beta}
\newcommand{\Za}{Z_\beta}
\newcommand{\ra}{\rho_\beta}
\newcommand{\n}{\nabla}
\newcommand{\vt}{\vartheta}
\newcommand{\its}{\int_{\{y=0\}}}
\newcommand{\Div}{\operatorname{div}}

\numberwithin{equation}{section}

\newcommand{\RN} {\mathbb{R}^N}
\newcommand{\Sob}{S^{1,p}(\Omega)}
\newcommand{\Dxk}{\frac{\partial}{\partial x_k}}
\newcommand{\Co}{C^\infty_0(\Omega)}
\newcommand{\py}{\text{lim}_{y \to 0} y^a \partial_ y}
\newcommand{\Je}{J_\ve}
\newcommand{\beq}{\begin{equation}}
\newcommand{\bea}[1]{\begin{array}{#1} }
\newcommand{\eeq}{ \end{equation}}
\newcommand{\ea}{ \end{array}}
\newcommand{\eh}{\ve h}
\newcommand{\Dxi}{\frac{\partial}{\partial x_{i}}}
\newcommand{\Dyi}{\frac{\partial}{\partial y_{i}}}
\newcommand{\Dt}{\frac{\partial}{\partial t}}
\newcommand{\aBa}{(\alpha+1)B}
\newcommand{\GF}{\psi^{1+\frac{1}{2\alpha}}}
\newcommand{\GS}{\psi^{\frac12}}
\newcommand{\HFF}{\frac{\psi}{\rho}}
\newcommand{\HSS}{\frac{\psi}{\rho}}
\newcommand{\HFS}{\rho\psi^{\frac12-\frac{1}{2\alpha}}}
\newcommand{\HSF}{\frac{\psi^{\frac32+\frac{1}{2\alpha}}}{\rho}}
\newcommand{\AF}{\rho}
\newcommand{\AR}{\rho{\psi}^{\frac{1}{2}+\frac{1}{2\alpha}}}
\newcommand{\PF}{\alpha\frac{\psi}{|x|}}
\newcommand{\PS}{\alpha\frac{\psi}{\rho}}
\newcommand{\ds}{\displaystyle}
\newcommand{\Zt}{{\mathcal Z}^{t}}
\newcommand{\XPSI}{2\alpha\psi \begin{pmatrix} \frac{x}{\left< x \right>^2}\\ 0 \end{pmatrix} - 2\alpha\frac{{\psi}^2}{\rho^2}\begin{pmatrix} x \\ (\alpha +1)|x|^{-\alpha}y \end{pmatrix}}
\newcommand{\Z}{ \begin{pmatrix} x \\ (\alpha + 1)|x|^{-\alpha}y \end{pmatrix} }
\newcommand{\ZZ}{ \begin{pmatrix} xx^{t} & (\alpha + 1)|x|^{-\alpha}x y^{t}\\
     (\alpha + 1)|x|^{-\alpha}x^{t} y &   (\alpha + 1)^2  |x|^{-2\alpha}yy^{t}\end{pmatrix}}
\newcommand{\norm}[1]{\lVert#1 \rVert}
\newcommand{\ve}{\varepsilon}
\newcommand{\D}{\operatorname{div}}
\newcommand{\G}{\mathscr{G}}

\newcommand{\tn}{\textnormal}
\newcommand{\mc}{\mathcal}
\newcommand{\mb}{\mathbb}
\newcommand{\ld}{\lambda}
\newcommand{\dd}{\partial}
\newcommand{\f}{\frac}
\newcommand{\nn}{\notag}
\newcommand{\sg}{\sigma}
\newcommand{\qd}{\quad}
\newcommand{\iy}{\infty}
\def\bu{\mathbf{u}}
\def\D{\tn{d}}
\def\bU{\mathbf{U}}
\def\Rb{\mb{R}}
\def\Lb{\mb{L}}
\def\Tbu{\tilde{\bu}}
\def\DTbu{\tn{div}\Tbu}
\def\bW{\mathbf{W}}
\def\Ib{\mathbb{I}}
\def\bP{\mathbf{P}}

\title[]{Extension problem for the  Fractional parabolic  Lam\'e operator and unique continuation}

\author{Agnid Banerjee}
\address{Tata Institute of Fundamental Research\\
Centre For Applicable Mathematics \\ Bangalore-560065, India}\email[Agnid Banerjee]{agnidban@gmail.com}

\author{Soumen Senapati}
\address{Tata Institute of Fundamental Research\\
Centre For Applicable Mathematics \\ Bangalore-560065, India}\email[Soumen Senapati]{soumen@tifrbng.res.in}

\thanks{A.B is supported in part by  Department of Atomic Energy,  Government of India, under
project no.  12-R \& D-TFR-5.01-0520.}

%
%
%
\keywords{}
\subjclass{35A02, 35B60, 35K05}

\maketitle
\begin{abstract}
In this paper, we introduce and analyse an explicit formulation of fractional powers of the parabolic Lam\'e operator $\mb{H}$ ( see \eqref{lameop1} below) and we then study the extension problem associated to such non-local operators. We also study the  various regularity properties of solutions to such an extension problem via   a  transformation   as in \cite{Gur, AITY, EINT, Alessandrini_Morassi_2001}, which reduces  the extension problem for the parabolic Lam\'e operator to another system that resembles  the extension problem   for the fractional heat operator. Finally in the case when $s \geq 1/2$,  by proving a  conditional doubling property for solutions to the corresponding reduced system  followed by a blowup argument,  we establish a  space-like strong unique continuation result  for $\mb{H}^s \bu=V\bu$.
\end{abstract}

\tableofcontents

\section{Introduction and statement of the main results} 

Let $\mu, \ld$ be two constants satisfying the following non-degeneracy condition for some $\delta_0>0$

\begin{align}\label{lgo}
    \mu \ge   \delta_0 \tn{ and}\ 2\mu+\ld \ge  \delta_0.
\end{align}

Corresponding to such constants,  we consider the isotropic Lam\'e operator in $\Rb^n,  n\ge 1$, defined by 
\begin{align}\label{Lame op}
    \mb{L}\bu =  \mu\Delta \bu + (\mu+\ld)\nabla \tn{div}\bu.
\end{align}
where $\bu:\mathbb R^n \to \R^n$ and is say  twice continuously differentiable.

Here $\mu$ and $\lambda$ are referred to as Lam\'e parameters. The Lam\'e operator appears quite prominently in the theory of elasticity ( see \cite{Val}). Quite interestingly, this operator appears  quite distinctly in the  Signorini problem of elasticity ( see \cite{Sig1, Sig2}). We refer to the  works \cite{An, Sch, RS}  for various regularity results for the  Signorini problem of elasticity  and other references can be found therein. 

In this work, for the parabolic Lam\'e operator \begin{equation}\label{lameop1}\mb{H} := \dd_t - \mb{L},\end{equation} we are interested in studying the fractional powers  $\mb{H}^s$ for $s \in (0,1)$. The parabolic Lam\'e operator has been  investigated for instance  in  \cite{Sh} where the  Dirichlet problem has been studied in Lipschitz domains.  Our goal is to give  a meaning framework of  an appropriate  extension problem for $\mb{H}^s$  in the spirit of \cite{Caffarelli_Sylvestre_2007, NS, ST}. We refer to  the recent work \cite{Sc} where the extension problem for the  time-independent fractional operator $(-\mb{L})^s$ has been extensively studied. We  also refer to the introduction in \cite{Sc} where various possible applications of such nonlocal systems in the theory  of continuum mechanics has been discussed.

Our main result  concerning the extension problem for $\mb{H}^s$ is as follows. See Section \ref{s:n} for the relevant notations.

\subsection{Main results}
\begin{thrm}\label{energy1}
Let $\bu\in H^{s}$ and $\Tbu$ be defined by
\begin{align}\label{defn of soln}
    \Tbu(x,y,t) :=  & \int_0^\iy \int_{\Rb^n} \bP^{(a)}_y(x,z,\tau) \bu(z,t-\tau) \D z\D \tau,
\end{align}
where $\bP^{(a)}_y$ is as in \eqref{Poisson kernel}.
Then $\Tbu$ solves the Dirichlet problem \eqref{solving extension problem} and satisfies 
\begin{align}
    \label{initial data} & \left\|\Tbu(\cdot,y,\cdot)-\bu(\cdot,\cdot)\right\|_{H^s} \rightarrow 0, \ \tn{ as } y\to 0+, \\
    \label{Neumann data} & \left\|\f{2^{2s-1}\Gamma(s)}{\Gamma(1-s)}y^{a}\dd_y \Tbu + \mb{H}^{s}\bu\right\|_{H^{-s}} \rightarrow 0, \ \tn{ as } y\to 0+, 
\end{align}
Furthermore for a fixed $M>0$, we have
\begin{align}   
    \label{energy} & \int_0^{M}\int_{\Rb^{n+1}} y^a |\Tbu|^2 \D x\D t\D y \preceq   M^{1+a}\|\bu\|^2_{L^2(\Rb^{n+1})}, \qd \int_0^\iy \int_{\Rb^{n+1}} y^a |\nabla_{x,y}\Tbu|^2 \D x\D t\D y \preceq \|\bu\|_{H^{s}}.
\end{align}
\end{thrm}
Theorem \ref{energy1} constitutes the parabolic counterpart of \cite[Theorem 8.5]{Sc}. We also mention over here that the extension problem for the fractional heat operator has been  independently developed in \cite{NS} and \cite{ST}.

We then study  the  qualitative properties of solutions to nonlocal equations of the type
\begin{equation}\label{maine}
\mb{H}^s \bu= V\bu,
\end{equation}
via such an  extension problem, where $V:\mathbb R^n \to \mathbb R$ is a scalar potential.  The various  regularity properties of solutions to the extension problem are obtained via a transformation inspired by \cite{Gur, AITY, EINT, Alessandrini_Morassi_2001}  which reduces the extension problem for $\mb{H}^s \bu= V\bu$  to a system of the following type 
\begin{equation}\label{reduced}
\begin{cases}
  y^a\dd_t {\bU}^* - \Div \left(y^a \nabla_{x,y}\bU^* \right) - y^a B \nabla {\bU^*} = 0\ \text{ $y>0$},
  \\   
   \py \bU^*=  \tilde V \bU^* \text{at $\{y=0\}$}.
  \end{cases}
\end{equation}
where $\tilde V$ is a matrix valued potential of the type  \eqref{matr} below which involves the first derivatives of $V$.

Now due to asymmetric nature of the potential matrix $\tilde V$ as in \eqref{matr},  unlike that in  scalar valued case as in \cite{ABDG, BG} where space-like and space-time strong unique continuation results  have been obtained for non-local  equations of the type
\[
(\partial_t - \Delta)^s= Vu,
\]
it is not obvious as to how   the Poon type frequency function approach as in \cite{Po} can be adapted  to obtain   conditional doubling properties or non-degeneracy estimates for solutions to the corresponding extension problem which then facilitates a blowup argument  that leads to the desired unique continuation property. This is a somewhat delicate aspect and we refer to subsection \ref{furred} for a further discussion on such an obstruction. However in the case when $s \geq 1/2$ and where $V$ satisfies  a  $C^{3}$ regularity assumption as in \eqref{vassump} below,  via a transformation  as in \eqref{W},  it turns out that one can get rid  of  the Neumann datum and then apply the methods in \cite{ABDG} to obtain the space-like strong unique continuation result that is stated in Theorem \ref{main} below.  Over here, we cannot stress enough the fact that although  the matrix potential $\tilde V$  of the reduced system \eqref{reduced} is not symmetric, nevertheless quite remarkably, it  has a specific structure that ensures that it commutes with its derivative matrices. This plays a key role in some of the  computations in  \eqref{W}-\eqref{calc2} below that is used to assert that the transformed function $W$ in \eqref{W} satisfies a differential inequality of the type \eqref{weq1}. It is to be noted that the precise structure of $\tilde V$ that we  obtain in \eqref{matr}  is due to the nature of the  reduction scheme that is employed to reduce the Lam\'e extension problem  to a fractional heat type extension system. This is indeed a new feature in the vectorial case and it also indicates   that  the question of strong unique continuation for nonlocal systems in general  is somewhat more subtle than for nonlocal scalar  equations.

We now state our unique continuation result. \begin{thrm}\label{main}
Let $s \in [1/2, 1)$ and $\bu: \R^n \to \R^n \in H^{s}$ be a solution to \eqref{maine}  in $B_1 \times (-1, 0]$ where the potential $V$ satisfies the growth assumption in \eqref{vassump} below. Then if $\bu$ vanishes to infinite order at $(0,0)$ in the following sense
\begin{equation}\label{vp}
\int_{B_r \times (-r^2, 0]} |\bu|^2 = O(r^k)\ \text{as $r \to 0$},
\end{equation}
for all $k \in \mathbb N$, 
then $\bu(\cdot, 0) \equiv 0$. 

\end{thrm}
\begin{rmrk}
For $\bu \in H^{s}$, the equation \eqref{maine} is to be understood in the $H^{-s}$ sense. 

\end{rmrk}

The precise assumption on the potential  $V$ is as follows:
\begin{equation}\label{vassump}
||V||_{C^{3}_{(x,t)}(\mathbb R^{n} \times \mathbb R)} \leq K,\ \text{for some $K \geq 0$.}\end{equation}

It is worth emphasizing that there is an example due to Frank Jones ( see \cite{Jo} of a global caloric function (i.e. solution to the heat equation) in $\mathbb R^n \times \mathbb R$ which is supported in a strip of the type $\mathbb R^n \times (t_1, t_2)$.  In view of such an example,  the space-like propagation of zeros  as claimed in Theorem \ref{main} is the best possible. In this connection, we mention that for local solutions to second order parabolic equations space-like strong unique continuation results were proven in the remarkable works \cite{EF, EFV}.

In order to provide a better perspective on our unique continuation result, we mention that for nonlocal elliptic equations  of the type $(-\Delta)^s +V$,  strong unique continuation results have been obtained in \cite{FF, FF1,  Ru, Ru1,  RW}.  The method in \cite{FF} combined the frequency function approach in \cite{GL, GL1} with the Caffarelli-Silvestre extension method in \cite{Caffarelli_Sylvestre_2007}. The approach in \cite{Ru, RW} is instead on Carleman estimates.

In the time dependent case, for global solutions of
\begin{equation}\label{fracheat1}
(\partial_t - \Delta)^s u =Vu,
\end{equation}
a backward space-time strong unique continuation theorem has been established in \cite{BG}  with appropriate assumptions on the potential $V$ which represents the non-local counterpart of the one first obtained by Poon in \cite{Po} for the local case $s=1$. See also \cite{AT} for related work on regularity of nodal sets of solutions to \eqref{fracheat1}.  More recently, a space like strong unique continuation result for local solutions to \eqref{fracheat1}   has been obtained in \cite{ABDG} which constitutes the nonlocal counterpart of the unique continuation results in the aforementioned works \cite{EF, EFV}. We also refer to \cite{BKS} for a unique continuation result for fractional powers of variable coefficient parabolic operators  and \cite{AB} for a related quantitative uniqueness result.

In closing, we refer to the works \cite{Alessandrini_Morassi_2001, AMR, AMRV, KLW, LNUW, LNW} for various qualitative and quantitative results on strong unique continuation for variable coefficient Lam\'e operators in the local case, i.e. when $\lambda$ and $\mu$ are functions of $x$.

The paper is organized as follows. In Section \ref{s:n}, we introduce some basic notations and notions and gather some preliminary results that relevant to this work. In Section \ref{s:fr1}, we define the fractional powers $\mb{H}^s$ via the Fourier transform  and then obtain a pointwise representation of such non-local operators in terms of   the associated semigroup. See Theorem \ref{representation} below. Section \ref{s:extension problem} is devoted to proving Theorem \ref{energy1}. In Section \ref{rl}, we obtain various quantitative regularity estimates for the associated extension problem that are needed for justifying  the Poon type computations as well as for   the blowup argument in the proof of the space-like strong unique continuation result Theorem \ref{main}.

\textbf{Acknowledgment:} We would like to thank Venky  Krishnan and Ramesh Manna for various useful discussions related to this project.

\section{Notations and Preliminaries}\label{s:n}
In this section we introduce the relevant notation and gather some auxiliary  results that will be useful in the rest of the paper. Generic points in $\Rn \times \R$ will be denoted by $(x_0, t_0), (x,t)$, etc. For an open set $\Omega\subset \Rn_x\times \R_t$ we indicate with $C_0^{\infty}(\Omega)$ the set of compactly supported smooth functions in $\Om$. We also indicate by $H^{\alpha}(\Omega)$ the non-isotropic parabolic H\"older space with exponent $\alpha$ defined in \cite[p. 46]{Li}. The symbol $\mathscr S(\R^{n+1})$ will denote the Schwartz space of rapidly decreasing functions in $\R^{n+1}$.

For $f\in L^1(\Rb^n)$, we denote by $\hat{f}$ its Fourier transform as below
\begin{align*}
   \mc{F}_{x\to\xi}{(f)} = \hat{f}(\xi) = \int_{\Rb^n} e^{-i x \cdot \xi} f(x) \D x
\end{align*}
Also then, we have the translation property which says
\begin{align*}
    \mc{F}_{x\to\xi}{(f(\cdot+h))} = e^{i h\cdot\xi} \mc{F}_{x\to\xi}{(f)}, \qd h \in \Rb^n.
\end{align*}
In the setting of the extension problem \eqref{solving extension problem},  we will deal with the thick half-space $\R^{n+1} \times \R^+_y$. At times it will be convenient to combine the additional variable $y>0$ with $x\in \Rn$ and denote the generic point in the thick space $\Rn_x\times\R^+_y$ with the letter $X=(x,y)$. For $x_0\in \Rn, y_0 \in \Rb$ and $r>0$ we let $B_r(x_0) = \{x\in \Rn\mid |x-x_0|<r\}$,
 $\mathbb B_r(X_0)= \{X = (x,y) \in \R^n \times \R \mid |x-x_0|^2 + |y-y_0|^2 < r^2\}$, $\mathbb B_r^+(x_0,y_0)= \mathbb B_r(X_0) \cap \{y>0\}$ (note that this is the upper half-ball),
and $\mathbb Q_r^+((x_0,t_0),0)=\mathbb B_r^+(x_0,0) \times [t_0,t_0+r^2)$.
When the center $x_0$ of $B_r(x_0)$ is not explicitly indicated, then we are taking $x_0 = 0$. Similar agreement for the thick half-balls $\mathbb B_r^+(x_0,0)$. Unless otherwise specified,  for notational ease, $\nabla U$ and  $\operatorname{div} U$ will respectively refer to the quantities  $\nabla_X U$ and $ \operatorname{div}_X U$.  The partial derivative in $t$ will be denoted by $\p_t U$ and also at times  by $U_t$. The partial derivative $\partial_{x_i} U$  will be denoted by $U_i$.
The distance between two points $(X_1, t_1)$ and $(X_2, t_2)$  in the space time is defined as
\begin{equation}\label{distance}
|(X_1, t_1) - (X_2, t_2)| \overset{def}= |X_1 - X_2| + |t_1 - t_2|^{1/2}.
\end{equation}

We will need the following coercivity result  in our analysis  which can be found in \cite[Lemma 5.1]{Giu}.
\begin{lemma}\label{coer}
Let $A=(A^{ij}_{\alpha\beta})$ be a constant matrix satisfying the following Legendre-Hadamard condition
\begin{equation}\label{lg}
\sum_{i,j,\alpha, \beta} A^{ij}_{\alpha\beta} \xi_i \xi_j \eta^\alpha \eta^\beta \geq \nu |\xi|^2 |\eta|^2
\end{equation}
with $\nu\geq 0$. Then one has for every $\tau \in C^{\infty}_0(\mathbb R^n)$
\begin{equation}\label{lg1}
\sum_{i,j, \alpha, \beta}\int A^{ij}_{\alpha \beta} \partial_\alpha \tau^i \partial_\beta \tau^j \geq \nu \int |\nabla \tau|^2.
\end{equation}

\end{lemma}

Such a coercivity result is crucial in proving the energy estimate in Theorem \ref{w22} for the Lam\'e extension problem.

In the final step of the proof of Theorem \ref{main}, when we analyse the blowup limit, we will need the following weak unique continuation result from \cite[Proposition 5.6]{LLR}. 

\begin{prop}\label{wucp}
Let $U_0$ be a  weak solution to
\begin{equation}\label{homog}
\begin{cases}
y^a \partial_t U_0 + \Div(y^a \nabla U_0) =0\ \text{ in  $\mathbb B_1^+ \times [0,1)$,}
\\
\py U_0((x,0), t) \equiv 0\ \text{for all $(x,t)  \in B_1 \times [0,1)$,}
\end{cases}
\end{equation}
such that $U_0((x,0), t) \equiv 0$ for all $(x,t) \in B_1 \times [0,1)$. 
Then $U_0 \equiv 0$ in $\mathbb B_1^+ \times [0,1)$.
\end{prop}
We mention over here that Proposition \ref{wucp} also follows from the space-like analyticity of solutions to \eqref{homog}  up to $\{y=0\}$ as established in the recent work \cite{BG1}.

\section{The fractional powers of the parabolic Lam\'e operator}\label{s:fr1}

For $f\in L^1(\Rb^n)$, we denote by $\hat{f}$ its Fourier transform as below
\begin{align*}
   \mc{F}_{x\to\xi}{(f)} = \hat{f}(\xi) = \int_{\Rb^n} e^{-i x \cdot \xi} f(x) \D x
\end{align*}
Also then, we have the translation property which says
\begin{align*}
    \mc{F}_{x\to\xi}{(f(\cdot+h))} = e^{i h\cdot\xi} \mc{F}_{x\to\xi}{(f)}, \qd h \in \Rb^n.
\end{align*}
From \eqref{Lame op}, it is not hard to notice that 
\begin{align*}
    \widehat{-\Lb\bu}(\xi) = \left( \mu |\xi|^2 I_n + (\mu+\ld) \xi \otimes \xi \right) \hat{\bu}(\xi)
\end{align*}
Having observed that, we define the heat semigroup on $\Rb^n$ with generator $-\Lb$ as 
\begin{align*}
   \widehat{P_{t} \bu}(\xi) = e^{-t \left(\mu|\xi|^2 I_n + (\mu+\ld) \xi \otimes \xi\right)} \hat{\bu}(\xi), \qd t>0. 
\end{align*}
At this point, we also introduce the evolutive semigroup $P^{\mb{H}}_{\tau}=\{e^{-\tau\mb{H}}\}_{\tau>0}$ on $\Rb^{n+1}$ as follows
\begin{align}\label{evolutive semigroup}
    \widehat{P^{\mb{H}}_{\tau} \bu}(\xi,\sg) = \widehat{e^{-\tau\mb{H}}\bu}(\xi,\sg) = e^{-\tau \left((\mu|\xi|^2 + i\sg)I_n + (\mu+\ld) \xi \otimes \xi\right)} \hat{\bu}(\xi,\sg).
\end{align}
From \eqref{evolutive semigroup}, one can verify that $\{P^{\mb{H}}_{\tau}\}_{\tau>0}$ defines a self-adjoint strongly continuous contraction semigroup on $L^2(\Rb^{n+1})$. Before proceeding further, let us mention an important identity which says 
\begin{align}\label{identity}
    e^{-t \left(\mu|\xi|^2 I_n + (\mu+\ld) \xi \otimes \xi\right)} = e^{-\mu|\xi|^2 t} I_n + \left(e^{-(2\mu+\ld)|\xi|^2 t} - e^{-\mu|\xi|^2 t}\right) \f{\xi\otimes\xi}{|\xi|^2}.
\end{align}
%
To prove the identity \eqref{identity}, we first write
\begin{align}\label{decomposition}
    \mu|\xi|^2 I_n + (\mu+\ld) \xi \otimes \xi = \mu|\xi|^2 \left(I_n - \f{\xi\otimes\xi}{|\xi|^2} \right) + |\xi|^2 (2\mu+\ld) \f{\xi\otimes\xi}{|\xi|^2}.
\end{align}
The fact that the matrix $\f{\xi\otimes\xi}{|\xi|^2}$ is idempotent implies that $\left(I_n - \f{\xi\otimes\xi}{|\xi|^2} \right)$ is also idempotent and 
\begin{align}\label{zero divisor}
    \left(I_n - \f{\xi\otimes\xi}{|\xi|^2} \right) \f{\xi\otimes\xi}{|\xi|^2} = \f{\xi\otimes\xi}{|\xi|^2} \left(I_n - \f{\xi\otimes\xi}{|\xi|^2} \right) = O_n
\end{align}
We now combine \eqref{decomposition} and \eqref{zero divisor} to obtain 
\begin{align*}
    e^{-t \left(\mu|\xi|^2 I_n + (\mu+\ld) \xi \otimes \xi\right)} & = e^{-t \mu|\xi|^2 \left(I_n - \f{\xi\otimes\xi}{|\xi|^2} \right) - t |\xi|^2 (2\mu+\ld) \f{\xi\otimes\xi}{|\xi|^2}} \\
    & = \sum\limits_{k=0}^{\iy} \f{(-t\mu|\xi|^2)^k}{k!} \left(I_n - \f{\xi\otimes\xi}{|\xi|^2} \right) + \sum\limits_{k=0}^{\iy} \f{\left(-t|\xi|^2 (2\mu+\ld) \right)^k}{k!} \f{\xi\otimes\xi}{|\xi|^2} \\
    & = e^{-t \mu|\xi|^2} \left(I_n - \f{\xi\otimes\xi}{|\xi|^2} \right) + e^{- t|\xi|^2 (2\mu+\ld)} \f{\xi\otimes\xi}{|\xi|^2} \\
    & = e^{-t \mu|\xi|^2} I_n + \left( e^{- t|\xi|^2 (2\mu+\ld)} - e^{-t \mu|\xi|^2} \right) \f{\xi\otimes\xi}{|\xi|^2},
\end{align*}
which proves \eqref{identity}.
In view of \eqref{identity}, we equivalently express
\begin{align*}
    \widehat{P_{t} \bu}(\xi) = e^{-t \mu|\xi|^2} \hat{\bu}(\xi) + \left(e^{- t|\xi|^2 (2\mu+\ld)} - e^{-t \mu|\xi|^2} \right) \f{\xi\otimes\xi}{|\xi|^2} \hat{\bu}(\xi).
\end{align*}
Following similar steps, we  also have 
\begin{align}\label{rep1}
     \widehat{P^{\mb{H}}_{\tau} \bu}(\xi,\sg) = e^{-\tau(\mu|\xi|^2+i\sg)} \hat{\bu}(\xi,\sg) + \left( e^{-\tau((2\mu+\ld)|\xi|^2+ i\sigma)} - e^{-\tau( \mu|\xi|^2 + i \sigma)} \right) \f{\xi\otimes\xi}{|\xi|^2} \hat{\bu}(\xi,\sg). 
\end{align}

\begin{definition}
 Let $s \in (0,1)$ and $\bu\in\mc{S}(\Rb^{n+1})$. We define 
\begin{align}\label{deffrac}
  \widehat{\mb{H}^s\bu}(\xi,\sg) =  \left((\mu|\xi|^2 + i\sg) I_n + (\mu+\ld) \xi \otimes \xi\right)^s \hat{\bu}(\xi,\sg).
\end{align}
\end{definition}
We note that the fractional power of the matrix 
\[A(\xi,\sg):= (\mu|\xi|^2 + i\sg) I_n + (\mu+\ld) \xi \otimes \xi\]
can also be realized in simple terms by means of diagonalization employed in \cite{Sc}. To see that, we first consider a rotation matrix $R(\xi)$ in $\Rb^n$ which takes the unit vector $\f{\xi}{|\xi|}$ to $e_1$. Here we assumed $\xi \neq 0$ and $e_1 = (1,0, \cdots, 0)$. Summarizing, we have 
\begin{align*}
    R^t(\xi) R(\xi) = I_n, \qd R(\xi)\xi = |\xi|e_1.
\end{align*}
Consequently, we express
\begin{align*}
    A(\xi,\sg) & = (\mu|\xi|^2 + i\sg) I_n + (\mu+\ld) |\xi|^2 \left(R^t(\xi)e_1 \otimes R^t(\xi)e_1\right) \\
    & = R^t(\xi) \left( (\mu|\xi|^2 + i\sg) I_n + (\mu+\ld) |\xi|^2 e_1 \otimes e_1 \right) R(\xi) \\
    & = R^t(\xi) \begin{pmatrix} 
         (2\mu+\ld)|\xi|^2 + i\sg & 0 & \dots  & 0 \\
         0 & (\mu|\xi|^2 + i\sg) & \dots  & 0 \\
            \vdots & \vdots & \ddots & \vdots \\
           0 & 0 & \dots  & (\mu|\xi|^2 + i\sg)
        \end{pmatrix}  R(\xi)
\end{align*}
and therefore we can write
\begin{align}\label{frac1}
    A^s(\xi,\sg) & = R^t(\xi) \begin{pmatrix} 
         \left((2\mu+\ld)|\xi|^2 + i\sg\right)^s & 0 & \dots  & 0 \\
         0 & (\mu|\xi|^2 + i\sg)^s & \dots  & 0 \\
            \vdots & \vdots & \ddots & \vdots \\
           0 & 0 & \dots  & (\mu|\xi|^2 + i\sg)^s
        \end{pmatrix}  R(\xi) \\
       & = \left(\mu|\xi|^2 + i\sg\right)^s I_n + \left( \left((2\mu+\ld)|\xi|^2 + i\sg\right)^s - \left(\mu|\xi|^2 + i\sg\right)^s \right) R^t(\xi) e_1 \otimes e_1 R(\xi)\notag \\
       & = \left(\mu|\xi|^2 + i\sg\right)^s I_n + \left( \left((2\mu+\ld)|\xi|^2 + i\sg\right)^s - \left(\mu|\xi|^2 + i\sg\right)^s \right) \f{\xi \otimes \xi}{|\xi|^2}.\notag
\end{align}
  Now for a given $\xi \neq 0$,  since the matrices $\frac{\xi \otimes \xi}{|\xi|^2}$ and $I_n - \frac{\xi \otimes \xi}{|\xi|^2}$ maps a vector into orthogonal components ( which follows from \eqref{zero divisor}), it turns out from the representation of the fractional powers as in \eqref{frac1} that
  \begin{equation}\label{nodegn}
 | \widehat{\mb{H}^s \bu} | \geq \text{min} ( |\mu|\xi|^2 + i\sigma|^s | \hat{\bu}|, |(2\mu+ \ld) |\xi|^2 + i \sigma|^s | \hat{\bu}|).
 \end{equation}
 From \eqref{nodegn} it thus follows that
the  natural domain for the definition  of $\mb{H}^s$ is 
\begin{align}\label{hs}
    \tn{Dom}(\mb{H}^s) = H^{2s}\overset{def}= \{\bu \in L^2(\Rb^{n+1}); (\mu|\xi|^2 + i\sg)^s \hat{\bu}(\xi,\sg) \in L^2(\Rb^{n+1})\}.
\end{align}
We now have the following Balakrishnan type representation for $\mb{H}^s$ based on Bochner subordination principle. See \cite{Ba}.
\begin{thrm}\label{representation}
For $s\in(0,1)$ and $\bu \in \mc{S}(\Rb^{n+1})$, we have 
\begin{align}\label{semigroup defn}
    \mb{H}^s\bu(x,t) = - \f{s}{\Gamma(1-s)} \int_0^\iy \left(P_\tau(\Lambda_{-\tau}u)(x,t) - u(x,t)\right) \f{\D\tau}{\tau^{1+s}}
\end{align}
where $\Lambda_h$ is the translation operator defined as $\Lambda_h u(x,t) := u(x,t+h), \ h \in \Rb$. 
\end{thrm}
\begin{proof}
We will consider fourier transform in both the space and time variables in order to prove the theorem. We observe that
\begin{align}
   \nn \mc{F}_{x,t}(P_\tau(\Lambda_{-\tau}u)(\xi,\sg) & = e^{-i\sg \tau} \mc{F}_{x} \left( P_\tau (\mc{F}_t u) \right)(\xi,\sg) \\
  \nn  & = e^{-i\sg \tau} e^{-\tau \left(\mu|\xi|^2 I_n + (\mu+\ld) \xi \otimes \xi\right)} \hat{\bu}(\xi,\sg) \\
    \label{multiplier} & = e^{-\tau \left( (\mu|\xi|^2 + i\sg) I_n + (\mu+\ld) \xi \otimes \xi\right)} \hat{\bu}(\xi,\sg).
\end{align}
Denoting RHS of \eqref{semigroup defn} by $h(x,t)$ and then taking its Fourier transform and also by using \eqref{rep1} we find 
\begin{align*}
    \widehat{h}(\xi,\sg) & = - \f{s}{\Gamma(1-s)} \int_0^\iy \left( e^{-\tau \left( (\mu|\xi|^2 + i\sg) I_n + (\mu+\ld) \xi \otimes \xi\right)} - I_n \right)\hat{\bu}(\xi,\sg) \f{\D\tau}{\tau^{1+s}} \\
    & = - \f{s}{\Gamma(1-s)} \int_0^\iy \left(e^{-\tau(\mu|\xi|^2+i\sg)} - 1\right) \hat{\bu}(\xi,\sg) \f{\D\tau}{\tau^{1+s}}  \\
    & \qd \qd \qd - \f{s}{\Gamma(1-s)} \int_0^\iy \left( e^{-\tau((2\mu+\ld)|\xi|^2 + i\sg)} - e^{-\tau (\mu|\xi|^2 + i \sg)} \right)  \f{\D\tau}{\tau^{1+s}} \f{\xi\otimes\xi}{|\xi|^2} \hat{\bu}(\xi,\sg) \\
    & = (\mu|\xi|^2+i\sg)^s \hat{\bu}(\xi,\sg) + \left( \left((2\mu+\ld)|\xi|^2 + i\sg\right)^s - \left(\mu|\xi|^2 + i\sg\right)^s \right) \f{\xi \otimes \xi}{|\xi|^2} \hat{\bu}(\xi,\sg) \\
    & = A^s(\xi,\sg) \hat{\bu}(\xi,\sg),
\end{align*}
where in the last inequality above, we used \eqref{frac1}.
\end{proof}


\section{The extension problem and energy estimates} \label{s:extension problem}
In this section, we build the framework for the related extension problem for $\mb{H}^s$.  Given $\bu$, we look  for $\Tbu$ satisfying
\begin{align}\label{solving extension problem}
\begin{cases}
   \dd_t \Tbu  = \left( \dd_y^2 + \f{a}{y}\dd_y + \Lb \right) \Tbu , \tn{ for } (x,t)\in\Rb^{n+1},\ y>0, \\
    \Tbu(x,t,0)  = \bu(x,t) , \tn{ for } (x,t)\in\Rb^{n+1},
\end{cases}
\end{align}
where 
\begin{equation}\label{relation}
a=1-2s.
\end{equation}
Here the Dirichlet data i.e. $\bu$ is prescribed on the boundary $y=0$. More precisely,  we show that  $\mb{H}^s$  can be realized as a certain weighted Dirichlet to Neumann map corrresponding to \eqref{solving extension problem} which  is precisely the content of Theorem \ref{energy1}.

A key ingredient in our analysis is the Poisson representation formula for $\Tbu$. To describe this, let $\bW$ denote the heat kernel associated to the Lam\'e operator $\Lb$ i.e.
\begin{align*}
    (\dd_t - \Lb)\bW(x-x_0,t-t_0) = \delta(x-x_0,t-t_0)\Ib_n.
\end{align*}
Then following \cite[Lemma A.1]{Sc}, we have
\begin{align}\label{FT of heat kernel}
    \mc{F}_{x} \bW(\xi,t) = e^{-\mu|\xi|^2t} \ \Ib_n + \left( e^{-(2\mu+\ld)|\xi|^2t} - e^{-\mu|\xi|^2t} \right) \f{\xi\otimes\xi}{|\xi|^2}.
\end{align}

Now  similar to that in \cite{BGMN}, for $y>0$, we define 
\begin{align}\label{Poisson kernel}
    \bP^{(a)}_y(x,z,t) = \f{1}{2^{2s}\Gamma(s)} \f{y^{2s}}{t^{1+s}} e^{-\f{y^2}{4t}} \bW(x-z,t)
\end{align}
which represents the Poisson kernel for the problem \eqref{solving extension problem}. For the sake of completeness, we briefly justify this claim.  
 From straightforward computations, we obtain 
\begin{align}\label{Poisson kernel 1}
    \left(\dd_y^2 + \f{a}{y}\dd_y\right) \bP^{(a)}_y(x,z,t) = \left(\f{y^2}{4t^2} - \f{1+s}{t}\right) \bP^{(a)}_y(x,z,t).
\end{align}
Now for $x \neq z$, using the fact that $\bW$ is the Heat kernel for $\mb{L}$, we have
\begin{align}\label{Poisson kernel 2}
    \left( \dd_t - \Lb \right) \bP^{(a)}_y(x,z,t) = \left(\f{y^2}{4t^2} - \f{1+s}{t}\right) \bP^{(a)}_y(x,z,t).
\end{align}
Combining \eqref{Poisson kernel 1} and \eqref{Poisson kernel 2}, we can conclude  that
\begin{align}\label{PDE for Poisson kernel}
    \dd_t \bP^{(a)}_y(x,z,t) = \left( \dd_y^2 + \f{a}{y}\dd_y + \Lb \right) \bP^{(a)}_y(x,z,t), \ \tn{ for } x \neq z, \ y>0.
\end{align}

We now proved the central result of this section which is Theorem \ref{energy1}.

\begin{proof}
We first verify that $\Tbu$ indeed solves the degenerate PDE in \eqref{solving extension problem}. 
To do so, we differentiate \eqref{defn of soln} under the integral sign and make use of the facts that $\bP^{(a)}_y$ satisfies \eqref{PDE for Poisson kernel} and
\begin{align*}
    \lim_{\tau\to 0+} \f{e^{-\f{y^2}{4\tau}}}{\tau^{1+s}} = \lim_{\tau\to \iy} \f{e^{-\f{y^2}{4\tau}}}{\tau^{1+s}} = 0,
\end{align*}
to arrive at 
\begin{align*}
    \dd_t \Tbu = \left( \dd_y^2 + \f{a}{y}\dd_y + \Lb \right) \Tbu, \qd \tn{ for } (x,t)\in\Rb^{n+1} \tn{ and } y>0.
\end{align*}
We now prove \eqref{initial data}, i.e.
\[ \lim\limits_{y\to 0+} \Tbu(\cdot,y,\cdot) = \bu(\cdot,\cdot),\ \tn{ in } H^s. \]
Taking Fourier transform w.r.t time variable in \eqref{defn of soln}, we have
\begin{align*}
    \mc{F}_t{\Tbu}(x,y,\sg) = \int_0^\iy \int_{\Rb^n} \bP^{(a)}_y(x,z,\tau) e^{-i\sg\tau} \mc{F}_{t}\bu(x,\sg) \D z\D\tau\ 
\end{align*}
which by subsequently taking the Fourier transform in $x$ variable further reduces to
\begin{align*}
    \widehat{\Tbu}(\xi,y,\sg) = \f{1}{2^{2s}\Gamma(s)} \int_0^\iy \f{y^{2s}}{\tau^{1+s}} e^{-\f{y^2}{4\tau}-i\sg\tau} \mc{F}_{x}\mb{W}(\xi,\tau) \widehat{\bu}(\xi,\sg) \D\tau.
\end{align*}
Now by using \eqref{FT of heat kernel} and \eqref{Poisson kernel} we find
\begin{align}
   \nn \widehat{\Tbu}(\xi,y,\sg) & = \f{y^{2s}}{2^{2s}\Gamma(s)} \int_0^\iy e^{- \f{y^2}{4\tau} + \left(-i\sg-\mu|\xi|^2\right)\tau} \f{\D\tau}{\tau^{1+s}} \widehat{\bu}(\xi,\sg) \\
   \label{FT of extended soln} & \qd + \f{y^{2s}}{2^{2s}\Gamma(s)} \int_0^\iy \left(e^{-\f{y^2}{4\tau} + \left(-i\sg-(2\mu+\ld)|\xi|^2\right)\tau} - e^{- \f{y^{2}}{4\tau}+\left(-i\sg-\mu|\xi|^2\right)\tau} \right) \f{\D\tau}{\tau^{1+s}} \f{\xi\otimes\xi}{|\xi|^2} \widehat{\bu}(\xi,\sg).
\end{align}
Now, we recall the following important identity 
\begin{align}\label{Macdonald identity}
    \int_0^\iy \tau^{\nu-1} e^{-\left(\f{\beta}{\tau}+\gamma \tau\right)} \D \tau = 2 \left(\f{\beta}{\gamma}\right)^{\f{\nu}{2}} \mc{K}_{\nu}(2\sqrt{\beta\gamma})
\end{align}
where $\tn{Re}(\beta), \tn{Re}(\gamma) > 0$ and $\mc{K}_\nu(\cdot)$ denotes the Macdonald's function. Using \eqref{Macdonald identity}, we find
\begin{align}
    \nn \int_0^\iy e^{- \f{y^2}{4\tau} + \left(i\sg-\mu|\xi|^2\right)\tau} \f{\D \tau}{\tau^{1+s}} & = \int_0^\iy \tau^{-1-s} e^{-\left(\f{y^2}{4\tau}+\left(\mu|\xi|^2+i\sg\right)\tau\right)} \D\tau \\
   \label{applying Macdonald 1} & = \f{2^{1+s}}{y^s} \left(\mu|\xi|^2+i\sg\right)^{\f{s}{2}} \mc{K}_s\left(y\sqrt{\mu|\xi|^2+i\sg}\right),
\end{align}
where we have used \eqref{Macdonald identity} for $\nu=-s,\ \beta= \f{y^2}{4},\  \gamma=\mu|\xi|^2+i\sg$ and the fact $\mc{K}_{s}=\mc{K}_{-s}$. Following an exactly similar set of arguments, we also obtain
    \begin{align}\label{applying Macdonald 2}
    & \int_0^\iy e^{-\f{y^2}{4\tau} - \left(i\sg+(2\mu+\ld)|\xi|^2\right)\tau} \f{\D \tau}{\tau^{1+s}} = \f{2^{1+s}}{y^s} \left((2\mu+\ld)|\xi|^2+i\sg\right)^{\f{s}{2}} \mc{K}_s\left(y\sqrt{(2\mu+\ld)|\xi|^2+i\sg}\right).
    \end{align}
In light of \eqref{applying Macdonald 1} and \eqref{applying Macdonald 2}, we obtain from \eqref{FT of extended soln}
\begin{align}
   \notag \widehat{\Tbu}(\xi,y,\sg) & = \f{2^{1-s}y^s}{\Gamma(s)} \left(\mu|\xi|^2+i\sg\right)^{\f{s}{2}} \mc{K}_s\left(y\sqrt{\mu|\xi|^2+i\sg}\right) \widehat{\bu}(\xi,\sg) \\
  \notag & \qd + \f{2^{1-s}y^s}{\Gamma(s)} \left( \left((2\mu+\ld)|\xi|^2+i\sg\right)^{\f{s}{2}} \mc{K}_s\left(y\sqrt{(2\mu+\ld)|\xi|^2+i\sg}\right) \right.\\
   \label{FT of extension} & \qd \qd \left. - \left(\mu|\xi|^2+i\sg\right)^{\f{s}{2}} \mc{K}_s\left(y\sqrt{\mu|\xi|^2+i\sg}\right)\right)\f{\xi\otimes\xi}{|\xi|^2} \widehat{\bu}(\xi,\sg).
\end{align}
We now recall the following property of the Macdoland's function 
\begin{align}\label{limit of Macdoland}
    \lim_{z\to 0} z^s\mc{K}_s(z) = 2^{s-1}\Gamma(s)
\end{align}
which will be important in the forthcoming steps.

We have using \eqref{FT of extension}
\begin{align}
   \label{proof of initial data}  & \|\Tbu(\cdot,y,\cdot)-u(\cdot,\cdot)\|_{H^s(\Rb^{n+1})} = \int_{\Rb^{n+1}} (1+||\xi|^2+i\sg|^2)^s|\widehat{\Tbu}(\xi,y,\sg)-\widehat{\bu}(\xi,\sg)|^2 \D\xi\D\sg \\
   \notag & \preceq \int_{\Rb^{n+1}} (1+||\xi|^2+i\sg|^2)^s \left\vert \f{2^{1-s}y^s}{\Gamma(s)} \left(\mu|\xi|^2+i\sg\right)^{\f{s}{2}} \mc{K}_s\left(y\sqrt{\mu|\xi|^2+i\sg}\right) - 1 \right\vert^2 |\widehat{\bu}(\xi,\sg)|^2 \\
   \notag &  + \int_{\Rb^{n+1}} (1+||\xi|^2+i\sg|^2)^s \left\vert y^s\left((2\mu+\ld)|\xi|^2+i\sg\right)^{\f{s}{2}} \mc{K}_s\left( y\sqrt{(2\mu+\ld)|\xi|^2+i\sg}\right) \right.\\
    \notag & \hspace{6cm}\left. - y^s\left(\mu|\xi|^2+i\sg\right)^{\f{s}{2}} \mc{K}_s\left(y\sqrt{\mu|\xi|^2+i\sg}\right)\right\vert^2 |\widehat{\bu}(\xi,\sg)|^2.
\end{align}

 We first  look at  the first term in RHS of \eqref{proof of initial data}. For $\epsilon>0$, let us choose $\delta>0$ small such that we have 
\[ \left\vert \f{2^{1-s}}{\Gamma(s)}z^s\mc{K}_s(z)- 1\right\vert \le \f{\epsilon}{2}, \tn{ when } |z|\le \delta\]
from \eqref{limit of Macdoland}. 
Next, we compute
\begin{align*}
    & \int_{\Rb^{n+1}} (1+||\xi|^2+i\sg|^2)^s \left\vert \f{2^{1-s}y^s}{\Gamma(s)} \left(\mu|\xi|^2+i\sg\right)^{\f{s}{2}} \mc{K}_s\left(y\sqrt{\mu|\xi|^2+i\sg}\right) - 1 \right\vert^2 |\widehat{\bu}(\xi,\sg)|^2 \D\xi\D\sg \\
    & \le \underbrace{\underset{|z|\le \delta}{\tn{sup}}\left\vert\f{2^{1-s}z^s}{\Gamma(s)}\mc{K}_s(z)-1\right\vert^2 \int_{\Rb^{n+1}} (1+||\xi|^2+i\sg|^2)^s |\widehat{\bu}(\xi,\sg)|^2 \D\xi\D\sg}_{\textnormal{(I)}}\\
    & + \underbrace{\underset{|z| > \delta}{\tn{sup}} \left\vert\f{2^{1-s}z^s}{\Gamma(s)}\mc{K}_s(z)-1\right\vert^2 \int_{\Rb^{n+1}} \mathbf{\chi}_{y\left|\mu|\xi|^2+i\sg\right|^{\f{1}{2}} > \delta} (1+||\xi|^2+i\sg|^2)^s |\widehat{\bu}(\xi,\sg)|^2 \D\xi\D\sg}_{\textnormal{(II)}}.
\end{align*}
 Using our choice of $\delta$, we can show that the term in $\textnormal{(I)}$ can be  upper bounded by  $\f{\epsilon}{2}\|u\|_{H^s(\Rb^{n+1})}$. To handle the term $\textnormal{(II)}$, we use boundedness of $z^s\mc{K}_s(z)$ when $z$ is away from zero. An application of Lebesgue dominated convergence theorem alongwith the fact that $u\in H^s(\Rb^{n+1})$ implies that
\begin{align*}
    \lim_{y\to 0+} \int_{\Rb^{n+1}} \mathbf{\chi}_{y\left|\mu|\xi|^2+i\sg\right|^{\f{1}{2}} > \delta} (1+||\xi|^2+i\sg|^2)^s |\widehat{\bu}(\xi,\sg)|^2 \D\xi\D\sg = 0.
\end{align*}
We can proceed similarly to establish
\begin{align*}
    & \lim\limits_{y\to 0+} \int_{\Rb^{n+1}} (1+||\xi|^2+i\sg|^2)^s \left\vert y^s\left((2\mu+\ld)|\xi|^2+i\sg\right)^{\f{s}{2}} \mc{K}_s\left( y\sqrt{(2\mu+\ld)|\xi|^2+i\sg}\right) \right.\\
    \notag & \hspace{2cm} \left. - y^s\left(\mu|\xi|^2+i\sg\right)^{\f{s}{2}} \mc{K}_s\left(y\sqrt{\mu|\xi|^2+i\sg}\right)\right\vert^2 |\widehat{\bu}(\xi,\sg)|^2 \D\xi\D\sg = 0.
\end{align*}
Hence we have proved \eqref{initial data}.
Now we proceed to prove \eqref{Neumann data}. Taking $y$ derivative in \eqref{FT of extension}, we find
\begin{align*}
    \dd_y \widehat{\Tbu}(\xi,y,\sg) & = \f{s}{y} \widehat{\Tbu}(\xi,y,\sg) + \f{2^{1-s}y^s}{\Gamma(s)} \left(\mu|\xi|^2+i\sg\right)^{\f{1+s}{2}} \mc{K}^{'}_s\left(y\sqrt{\mu|\xi|^2+i\sg}\right) \widehat{\bu}(\xi,\sg) \\
  \notag &  + \f{2^{1-s}y^s}{\Gamma(s)} \left( \left((2\mu+\ld)|\xi|^2+i\sg\right)^{\f{1+s}{2}} \mc{K}^{'}_s\left(y\sqrt{(2\mu+\ld)|\xi|^2+i\sg}\right) \right.\\
    & \left. - \left(\mu|\xi|^2+i\sg\right)^{\f{1+s}{2}} \mc{K}^{'}_s\left(y\sqrt{\mu|\xi|^2+i\sg}\right)\right)\f{\xi\otimes\xi}{|\xi|^2} \widehat{\bu}(\xi,\sg).
\end{align*}
For notational convenience, we denote 
\begin{align*}
    L_1 = \sqrt{\mu|\xi|^2+i\sg}, \qd L_2 = \sqrt{(2\mu+\ld)|\xi|^2+i\sg}.
\end{align*}
Now we use the following recurrence and derivative formula for the Macdoland's function 
\begin{align}\label{recu}
   \mc{K}^{'}_s(z)= \f{s}{z} \mc{K}_s(z) - \mc{K}_{s+1}(z),\qd \mc{K}_{s+1}(z) - \mc{K}_{s-1}(z) = \f{2s}{z} \mc{K}_s(z).
\end{align}
Using the first recurrence relation in \eqref{recu} above, we have
\begin{align}
   \label{derivative of extended soln}  \dd_y \widehat{\Tbu}(\xi,y,\sg) & =  \f{2^{1-s}y^s}{\Gamma(s)} \left( \f{2s}{y} L_1^{s} \mc{K}_s(L_1y) - L_1^{1+s}\mc{K}_{1+s}(L_1y)\right) \widehat{\bu}(\xi,\tau) \\
    \nn & \qd + \f{2^{1-s}y^s}{\Gamma(s)} \left\{ \left( \f{2s}{y} L_2^{s} \mc{K}_s(L_2y) - L_2^{1+s}\mc{K}_{1+s}(L_2y) \right) \right. \\
    \nn  & \qd \qd \qd + \left. \left(\f{2s}{y} L_1^{s} \mc{K}_s(L_1y) - L_1^{1+s}\mc{K}_{1+s}(L_1 y) \right) \right\} \f{\xi\otimes\xi}{|\xi|^2} \widehat{\bu}(\xi,\sg). 
\end{align}
Now using the second recurrence relation in \eqref{recu} we find
\begin{align*}
    & \f{\Gamma(s)}{2^{1-s}} y^{1-2s} \dd_y \widehat{\Tbu}(\xi,y,\sg) \\
    & = - L_1^{1+s} y^{1-s} \mc{K}_{1-s}(L_1y) \widehat{\bu}(\xi,\tau) - \left\{L_2^{1+s} y^{1-s} \mc{K}_{1-s}(L_2 y) - L_3^{1+s} y^{1-s} \mc{K}_{1-s}(L_3y)  \right\} \f{\xi\otimes\xi}{|\xi|^2} \widehat{\bu}(\xi,\sg) \\
    & \underset{y\to 0+}{\longrightarrow} - 2^{-s} \Gamma(1-s) L_1^{2s} \widehat{\bu}(\xi,\tau) - 2^{-s} \Gamma(1-s) \left(L_2^{2s} - L_1^{2s}\right) \f{\xi\otimes\xi}{|\xi|^2} \widehat{\bu}(\xi,\sg) = - \f{\Gamma(1-s)}{2^{s}} \widehat{\mb{H}^s\bu}(\xi,\sg).
\end{align*}
In the last equality  above, we used \eqref{frac1}.  Along with the above pointwise limit, we can imitate the arguments in the proof of \eqref{initial data} using duality ( see for instance the proof of Theorem 3.1 in \cite{BKS})  to guarantee $\lim\limits_{y\to 0+} \f{2^{2s-1}\Gamma(s)}{\Gamma(1-s)} y^a \dd_y \Tbu =  (\dd_t -\Lb)^s\bu$ in $H^{-s}$ topology.

Next we turn our attention to proving \eqref{energy}. Using the Fourier representation of $\Tbu$ from \eqref{FT of extended soln} and Plancherel theorem, we easily note that
\begin{align*}
   \|\Tbu(\cdot,y,\cdot)\|^2_{L^2(\Rb^{n+1})} \preceq y^{2s} \int_{\Rb^{n+1}} \int_0^\iy  \f{e^{-\f{y^2}{4\tau}}}{\tau^{1+s}} \D\tau\ |\widehat{\bu}(\xi,\sg)|^2 \D\xi\D\sg \preceq \|\bu\|^2_{L^2(\Rb^{n+1})}
\end{align*}
which results in 
\begin{align*}
    \int_{0}^M \int_{\Rb^{n+1}} y^a |\Tbu|^{2} \D x\D y\D t \preceq \|\bu\|^2_{L^2(\Rb^{n+1})} M^{1+a}.
\end{align*}
Now we want to show that $\int_0^\iy \int_{\Rb^{n+1}} y^a |\nabla_{x,y}\Tbu|^2 \D x\D t\D y \preceq \|\bu\|_{H^{s}}$. In order to do so, we again use the Plancherel theorem, \eqref{FT of extension} and \eqref{derivative of extended soln}  to obtain 
\begin{align}
   \nn \int_0^\iy \int_{\Rb^{n+1}} y^a |\nabla_{x,y}\Tbu|^2 \D x\D t\D y & = \int_0^\iy \int_{\Rb^{n+1}} y^a \left( |\xi|^2 |\widehat{\Tbu}|^2 + \left|\dd_y \widehat{\Tbu}\right|^2 \right)(\xi,y,\sg) \ \D\xi\D\sg\D y \\
   \nn & \preceq \int_0^\iy \int_{\Rb^{n+1}} y \left( |\xi|^2 L_1^{2s} \mc{K}^2_s(L_1y) + |\xi|^2 L_2^{2s} \mc{K}^2_s(L_2y) \right) |\widehat{\bu}(\xi,\sg)|^2\D\xi\D\sg\D y \\
   \nn & + \int_0^\iy \int_{\Rb^{n+1}} y \left( L_1^{2+2s} \mc{K}^2_{1-s}(L_1y) + L_2^{2+2s} \mc{K}^2_{1-s}(L_2y) \right) |\widehat{\bu}(\xi,\sg)|^2\D\xi\D\sg\D y \\
   \label{gradient estimate} & \preceq \int_0^\iy \int_{\Rb^{n+1}} \sum_{i=1}^2  y L_i^{2+2s} \left( \mc{K}^2_s(L_i y) + \mc{K}^2_{1-s}(L_i y)\right) |\widehat{\bu}(\xi,\sg)|^2 \D\xi\D\sg\D y. 
\end{align}
In the very last step, we have used $|\xi| \le L_i,\ i=1,2$. Now similar to \cite[Lemma 4.5]{BG}, we can use the asymptotics of Bessel functions  to  obtain
\begin{align}\label{asymp}
    \int_0^\iy y L_i^2 \left( \mc{K}^2_s(L_i y) + \mc{K}^2_{1-s}(L_i y) \right) \D y \le C_s, \qd i=1, 2. 
\end{align}
Using \eqref{asymp} in \eqref{gradient estimate} we thus conclude
\begin{align*}
    \int_0^\iy \int_{\Rb^{n+1}} y^a |\nabla_{x,y}\Tbu|^2 \D x\D t\D y \preceq \sum_{i=1}^2 \int_{\Rb^{n+1}}  L_i^{2s} |\widehat{\bu}(\xi,\sg)|^2\D\xi \D\sg \preceq \|\bu\|_{H^s}.
\end{align*}
This finishes the proof of the theorem.
\end{proof}

\section{Regularity theory for the extension problem}\label{rl}
It follows from Theorem  \ref{energy1} that if $\bu \in H^{s}$ solves $\mb{H}^s \bu\overset{def}= (\partial_t - \mb{L})^s \bu = V\bu$, then we have  by analogous arguments as in the proof of Lemma 4.6 in \cite{BKS}  that $\Tbu$ is a weak solution to the following extension problem
\begin{equation}\label{ext1}
\begin{cases}
   \dd_t \Tbu = \left( \dd_y^2 + \f{a}{y}\dd_y + \Lb \right) \Tbu , & \tn{ for } y>0, \\
    \Tbu(x,t,0)=\bu(x,t), & \tn{ for } (x,t)\in\Rb^{n+1},
    \\
    \py \Tbu= V \bu & \text{in $B_1 \times (-1, 0]$}.
\end{cases}    
\end{equation}
We refer to Definition 4.3 in \cite{BG} for the precise notion of weak solutions. See also \cite{BS}. We now show the smoothness of weak solutions to \eqref{ext1}. As an intermediate step, we prove the following $W^{2,2}$ type estimate.
 
\begin{thrm}\label{w22}
Let $\Tbu$ be a weak solution to \eqref{ext1} in $\mathbb B_1^+ \times (-1, 0]$ where $V$ satisfies \eqref{vassump}.  Then it follows that $\nabla \nabla_x \Tbu \in L^{2}_{loc}(\mathbb B_1^+ \times (-1, 0],  y^a dXdt).$
\end{thrm}
\begin{proof}
It suffices to show that the  following estimate holds
\begin{equation}\label{des}
\int_{\mathbb B_{1/2}^+ \times (-1/4, 0]} \left( | \nabla \Tbu|^2 + |\nabla \nabla_x \Tbu|^2  \right)  y^a dXdt \leq C \int_{\mathbb B_1^+\times (-1, 0]} |\Tbu|^2 y^a dXdt,
\end{equation}
where $C$ depends on $\delta_0$ in \eqref{lgo} and the constant in \eqref{vassump}.

We first note that $\Tbu$ solves the following parabolic system in $\{y >0\}$ for $i=1,..., n$
\begin{equation}\label{ei}
\sum_{j, \alpha, \beta=1}^nA^{ij}_{\alpha\beta} \Tbu^{j}_{\alpha \beta}  + \left(\f{a}{y}\dd_y + \dd_y^2\right) \Tbu^{i} = \Tbu^{i}_t,\end{equation}
where $$ A^{ij}_{\alpha \beta}= \left(\mu \delta^{ji}_{\alpha \beta} + (\mu +\lambda) \delta^{\alpha i}_{\beta j} \right).$$ Now using \eqref{ei} and the boundary condition in \eqref{ext1},  from the proof of Theorem 5.1 in \cite{BG}   we find that the following inequality estimate holds for all $r<1$
\begin{equation}\label{inter1}
\int_{\mathbb B_{1}^+ \times (-1, 0]}  |(\phi\Tbu)_y|^2 y^a + \int_{\mathbb B_1^+ \times (-1, 0]} \sum_{i,j, \alpha, \beta=1}^n A^{ij}_{\alpha\beta} (\phi \Tbu^{j})_{\alpha} (\phi \Tbu^{i})_{\beta} y^a\leq \frac{C}{(1-r)^2} \int_{\mathbb B_1^+ \times (-1, 0]} |\Tbu|^2y^a,
\end{equation}

where $\phi$ is a suitable cut-off such that $\phi \equiv 1$ in $\mathbb B_r \times (-r^2, 0]$ and vanishes outside $\mathbb B_{1} \times (-1, 0]$.

Then from \eqref{lgo}, we note  that the constant matrix $(A^{ij}_{\alpha\beta})$ satisfies the Legendre-Hadamard condition \eqref{lg} with $\nu=\delta_0$.  Thus using the coercivity result in Lemma \ref{coer} we deduce the  following energy estimate from \eqref{inter1} for all $r <1$\begin{equation}\label{en1}
\int_{\mathbb B_r^+ \times (-r^2, 0]} |\nabla \Tbu|^2 y^a dXdt  \leq \frac{C}{(1-r)^2} \int_{\mathbb B_1^+ \times (-1, 0]} |\Tbu|^2 y^a.
\end{equation}
Moreover  for $i=1,...,n$, if we let \begin{equation}\label{diff} \tau_{h, i} \Tbu(X,t)= \frac{\Tbu(X+ he_i,t) - \Tbu(X,t)}{h}, \end{equation} then we note that $\bf v=\tau_{h,i} \Tbu$ solves the following problem

\begin{equation}\label{ext2}
\begin{cases}
  \sum_{j, \alpha, \beta=1}^nA^{ij}_{\alpha\beta} \bf v^{j}_{\alpha \beta}  + \left(\f{a}{y}\dd_y + \dd_y^2\right) \bf v^{i} = \bf{ v}^{i}_t , & \tn{ for $i=1,..., n$ and } y>0, \\
    \py \bf v=( \tau_{h, i} V) \bu + V(\cdot+  he_i,\cdot) \bf v   & \text{in $B_1 \times (-1, 0]$}.
\end{cases}    
\end{equation}
Since $V$ satisfies \eqref{vassump}, we find that $\tau_{h, i} V$ is bounded independent of $h$. Moreover we have  that $\bf v$ satisfies the same equation as $\bu$ in the "bulk" (i.e. in $\{y >0\})$. Therefore it again follows from the proof of Theorem 5.1 in \cite{BG} using also the coercivity result in Lemma \ref{coer} that $\tau_{h, i} \bu= \bf v$ satisfies the following energy estimate 
\begin{equation}\label{en2}
\int_{\mathbb B_{1/2}^+ \times (-1/4, 0]} |\nabla \tau_{h, i} \Tbu |^2 y^a dXdt  \leq C  \int_{\mathbb B_{3/4}^+ \times (-9/16, 0]} \left(  |\tau_{h, i} \Tbu|^2  + |\Tbu|^2 + |\nabla \Tbu|^2 \right) y^a dXdt.
\end{equation}
Now for all $h$ small enough, it follows that

\[
\int_{\mathbb B_{3/4}^+ \times (-9/16, 0]} |\tau_{h, i} \Tbu|^2 y^a dXdt  \leq C \int_{\mathbb B_{5/6}^+ \times (-25/36, 0]} |\nabla \Tbu|^2 y^a.\]
Using this inequality along with \eqref{en1} in \eqref{en2}, we can finally assert the following estimate  holds
\begin{equation}\label{fn0}
\int_{\mathbb B_{1/2}^+ \times (-1/4, 0]}  |\nabla \tau_{h, i} \Tbu|^2 y^a dX dt \leq C  \int_{\mathbb B_1^+ \times (-1, 0]} |\Tbu|^2 y^a.
\end{equation}
By letting $h \to 0$ in \eqref{fn0}, we find that \eqref{des} follows.

\end{proof}
\subsection*{Reduction to a fractional heat  type extension problem}

In order to obtain further regularity of solutions,  we   use a reduction technique in our analysis which is borrowed from \cite{Gur, AITY, EINT, Alessandrini_Morassi_2001}.  Such a reduction  is  also crucial for the proof of our space-like strong unique continuation result in Theorem \ref{main}.

Note that $\Tbu$ solves the extension problem \eqref{ext1}.
Now  by formally applying the divergence operator  (with respect to $x$ variable) in \eqref{ext1}, we find
\begin{align}
\begin{cases}\label{extended parabolic Lame for div(u)}
   \dd_t (\DTbu) = \left( \dd_y^2 + \f{a}{y}\dd_y\right) \DTbu + \tn{div}\Lb(\Tbu) , & \tn{ for } y>0, \\
    \DTbu(x,t,0)=\tn{div}\bu(x,t), & \tn{ for } (x,t)\in\Rb^{n+1}.
\end{cases}    
\end{align}
We then compute the term $\tn{div}\Lb(\Tbu)$ in the following way
\begin{align}\label{divL} 
   \tn{div}\Lb(\Tbu) = \mu \Delta (\DTbu) + (\mu+\ld)\Delta(\DTbu) 
  = (2\mu+\ld) \Delta(\DTbu).
\end{align}
Before proceeding further, we would like to alert the reader that in this subsection, $\Div \Tbu$ will refer to the tangential divergence $\Div_x \Tbu$.

Now  let
\begin{align*}
    \bU = \begin{pmatrix}
          \Tbu \\
          \tn{div} \Tbu
          \end{pmatrix}
          := \begin{pmatrix}
          \bU^1 \\
          \bU^2 \\
          \vdots\\
          \bU^{n+1}
          \end{pmatrix}.
\end{align*} 
Then by combining \eqref{ext1} and \eqref{divL} we observe that
\begin{align}\label{AM transform}
    \left( \dd_t  - \left( \dd_y^2 + \f{a}{y}\dd_y\right) \right) \bU = B_0 \Delta \bU + B_1 \nabla\bU.
\end{align}
where  $B_0$ and $B_1$ are as follows%
\begin{align*}
    B_0 = \begin{pmatrix}
          \mu & 0 & 0 & \dots  & 0 \\
          0 & \mu & 0 & \dots  & 0 \\
           \vdots & \vdots & \vdots & \ddots & \vdots \\
          0 & 0 & 0 & \dots  & (2\mu+\ld)
          \end{pmatrix}, \qd  
     B_1\nabla\bU = \begin{pmatrix} 
                     (\mu+\ld)\dd_{x_1}\bU^{n+1} \\
                     (\mu+\ld)\dd_{x_2}\bU^{n+1} \\
                    \vdots \\
                    0
                   \end{pmatrix}    . 
\end{align*}
Moreover, formally one has
\begin{align*}
    \bU(x,t,0) = \begin{pmatrix}
          \bu \\
          \tn{div} \bu
          \end{pmatrix}.
\end{align*}
Now taking divergence w.r.t $x$ variables we find
\begin{align*}
    \lim_{y\to 0+} y^a \dd_y (\tn{div}\bu) =  V(x,t)\tn{div}\bu +  \nabla_x V(x,t) \cdot\bu
\end{align*}
As a result, we have
\begin{align}\label{matr}
    \lim_{y\to 0+} y^a \dd_y \bU =  \begin{pmatrix}
          V(x,t) & 0 & 0 & \dots  & 0 \\
          0 & V(x,t) & 0 & \dots  & 0 \\
           \vdots & \vdots & \vdots & \ddots & \vdots \\
          \dd_1 V & \dd_2 V & \dd_3 V & \dots  & V(x,t)
          \end{pmatrix} \times  \bU= \hat{V} \bU.
\end{align}
To reduce \eqref{AM transform} into a further accessible form, we make the following change of variables 
\begin{align}\label{star}
    {\bU}^{*}(x,y,t) = \begin{pmatrix}
          \bU^1\left(\sqrt{\mu}x_1,\dots,\sqrt{\mu}x_n,y, t\right) \\
         \bU^2\left(\sqrt{\mu}x_1,\dots,\sqrt{\mu}x_n,y,t\right)\\
          \vdots\\
        \bU^{n+1}\left(\sqrt{2\mu+\ld}x_1,\dots,\sqrt{2\mu+\ld}x_n,y,t\right).
          \end{pmatrix}
\end{align}
We then notice that in $\{y>0\}$, $\bU^{\star}$ solves
\begin{align*}
    y^a\dd_t {\bU}^{\star} - \Div \left(y^a \nabla_{x,y}\bU^{\star}\right) - y^a B \nabla {\bU}^{\star} = 0,
\end{align*}
where

\begin{align}\label{drift1}
    B\nabla\bU^* = \begin{pmatrix} 
                     \f{(\mu+\ld)}{\sqrt{\mu}}\dd_1(\bU^*)^{n+1} \\
                     \f{(\mu+\ld)}{\sqrt{\mu}}\dd_2(\bU^*)^{n+1} \\
                    \vdots \\
                    0
                   \end{pmatrix}. 
\end{align}
Moreover, the weighted Neumann condition also gets transformed as 
\begin{equation}\label{dert}
\py \bU^*= \tilde V \bU^*,
\end{equation}

where  $\tilde V$ has a similar structure as the matrix valued function in \eqref{matr}. Now from the twice Sobolev  differentiability result  as  in Theorem \ref{w22}, we find that the formal computations  in \eqref{extended parabolic Lame for div(u)}-\eqref{dert} above can be justified  by weak type arguments and we thus have the following   result on the reduction of the  parabolic Lam\'e extension problem to an almost decoupled parabolic system.

\begin{lemma}\label{red1}
Let  $\bU^{*}$ be as in \eqref{star} corresponding to $\Tbu$ which solves the  extension problem \eqref{ext1}. Then $\bU^*$  is a weak solution to the following problem

\begin{equation}\label{staru}
\begin{cases}
  y^a\dd_t {\bU}^* - \Div \left(y^a \nabla_{x,y}\bU^* \right) - y^a B \nabla {\bU^*} = 0\ \text{in $\mathbb B_1^+ \times (-1, 0]$},
  \\   
   \py \bU^*=  \tilde V \bU^* \text{in $B_1 \times (-1, 0]$}.
  \end{cases}
\end{equation}
\end{lemma}

 Further regularity for the parabolic Lam\'e extension problem \eqref{ext1}  will follow from the  regularity results for \eqref{staru} that we develop subsequently.

In this direction, we first prove the following H\"older continuity result  for a general class of extension problem modelled on \eqref{staru} via compactness arguments similar to that employed in \cite{BS}.

\begin{thrm}\label{hold}
Let $\bU$ be a weak solution to the following problem
\begin{equation}\label{res1}
\begin{cases}
  y^a\dd_t {\bU} - \Div \left(y^a \nabla_{x,y}\bU \right) - y^a B \nabla {\bU} = 0\ \text{for $y>0 \cap \mathbb B_1^+ \times (-1, 0]$},
  \\   
   \py \bU=  W \bU + W_1\ \text{in $B_1 \times (-1, 0]$},
  \end{cases}
\end{equation}
where $B, W, W_1$ are bounded. Then there exists $\alpha= \alpha(a)$ such that $\bU \in H^{\alpha}(\overline{\mathbb B_{1/2}^+} \times (-1/4, 0])$. 
\end{thrm}
\begin{proof}
\emph{Step 1:}
\emph{Claim 1:} We first show that given $\ve>0$, there exists $\delta>0$ such that if
\begin{equation}\label{sm1}
\begin{cases}
\int_{\mathbb B_1^+ \times (-1, 0]} |\bU|^2 y^a \leq 1, 
\\
|B|, |W|, |W_1| \leq \delta,
\end{cases}
\end{equation}
then there exists $U_0$ which solves  the fractional heat extension system
\begin{equation}\label{res2}
\begin{cases}
  y^a\dd_t U_0 - \Div \left(y^a \nabla_{x,y}U_0 \right)  = 0\ \text{for $y>0 \cap \mathbb B_{1/2}^+ \times (-1/4, 0]$},
  \\   
   \py U_0=  0\ \text{in $B_{1/2} \times (-1/4, 0]$},
  \end{cases}
\end{equation}
such that
\begin{equation}\label{close}
\int_{\mathbb B_{1/2}^+ \times (-1/4, 0]} |\bU - U_0|^2 y^a \leq \ve.
\end{equation}
We argue by contradiction. If not, there exists an $\ve_0>0$ such that no choice of $\delta$ works.  That means for each $k=1,2,...$, there exists $U^k$ which solves \eqref{res1} corresponding $B^k, W^k, W^k_1$ satisfies 
\begin{equation}\label{s1}
\begin{cases}
\int_{\mathbb B_1^+ \times (-1, 0]} |U^k|^2 y^a \leq 1,
\\
|B^k|, |W^k|, |W^k_1| \leq \frac{1}{k},
\end{cases}
\end{equation}
 such that  for all $k$, we have
 \begin{equation}\label{nclose}
 \int_{\mathbb B_{1/2}^+ \times (-1/4, 0]} |U^k - U_0|^2 y^a > \ve_0
 \end{equation}
 for all such $U_0$ which solves \eqref{res2}. 
Now from the proof of Theorem 5.1 in \cite{BG}   and the bounds in \eqref{s1}, it follows that 
 \begin{equation}\label{s2}
 \int_{\mathbb B_{1/2}^+ \times (-1/4, 0]} |\nabla U^k|^2 y^a \leq C.
 \end{equation}
 Moreover from \eqref{s1} and \eqref{s2}, it follows in a standard way ( see for instance the proof of Theorem  4.1 in \cite{AK}) that $\partial_ t U^k \in L^{2}((-1/4, 0]; (W^{1,2}(\mathbb B_{1/2}^+, y^a dX))^*)$ are uniformly bounded independent of $k$. Here $(W^{1,2}(\mathbb B_{1/2}^+, y^a dX))^*$ denotes the dual sapce.  Therefore  by  the Aubin-Lions compactness lemma  ( \cite{Au}), we have that up to a subsequence $k \to \infty$, $U^k$'s converge to some $U^0$ in $L^{2}(\mathbb B_{1/2}^+ \times (-1/4, 0])$ which by standard weak type arguments, is a weak solution to \eqref{res2}. This contradicts \eqref{nclose} for large enough $k$'s and proves \emph{Claim 1}. 
 
 \emph{Step 2:}
 We now show that there exists  universal $r, \delta, \alpha \in (0,1)$ such that  if \eqref{sm1} holds, then there exists a vector $b_1$ with universal bounds such that
 \begin{equation}\label{sm2}
 \int_{\mathbb B_r^+ \times (-r^2, 0]} |\bU - b_1|^2 y^a \leq r^{n+3+a+2\alpha}.
 \end{equation}
 From \emph{Step 1}, we obtain the existence of $U_0$. Since $U_0$ is smooth( see for instance \cite[Theorem 1.1]{BG1}), we have that 
 the following inequality holds
 \begin{equation}\label{sm4}
 \int_{\mathbb B_r^+ \times (-r^2, 0]} |U_0 - U_0(0, 0))|^2 y^a \leq C r^{n+3+a+ 2}.
 \end{equation}  
 We now choose $\alpha$ such that 
 \begin{equation}\label{kj}
 \alpha < \text{min}( 2s, 1).\end{equation}
  We then choose $r< 1/2$ such that 
 \begin{equation}\label{sml}
C r^{n+3+a+ 2} = \frac{1}{4} r^{n+3+a +2\alpha}.
\end{equation}
Subsequently we let
\[
\ve= \frac{1}{4} r^{n+3+a +2\alpha}
\]
which decides the choice of $\delta$. We thus have
\begin{equation}\label{kj1}
\int_{\mathbb B_{r}^+ \times (-r^2, 0]} |\bU - U_0|^2 y^a \leq \frac{1}{4} r^{n+3+a +2\alpha}.
\end{equation}
\eqref{sm2} thus follows from \eqref{sm4} and \eqref{kj1}  with $b_1= U_0(0,0)$ in view of the fact that \eqref{sml} holds.

\emph{Step 3:}
With the smallness assumption as in \emph{Step 1} and \emph{Step 2}, we now show that for $k=1,2,..$,  we have that there exists $b_k$ such that
\begin{equation}\label{sm5}
 \int_{\mathbb B_{r^k}^+ \times (-r^{2k}, 0]} |\bU - b_k|^2 y^a \leq r^{k(n+3+a+2\alpha)},
 \end{equation}
 such that 
 \begin{equation}\label{bk}
 |b_k -  b_{k+1}| \leq C r^{k\alpha}.
 \end{equation}
 For $k=1$, it is as in \emph{Step 2}. We assume that \eqref{sm5} holds up to $k$. We now show this implies that \eqref{sm5} holds for $k+1$. 
 By letting
 \[
 \tilde U= \frac{\bU(r^kX, r^{2k} t)- b_k}{r^{k\alpha}}
 \]
 we find that $\tilde U$ solves  \eqref{res1} corresponding to 
 \[
 \tilde B(X,t)= r^{k}B(r^kX, r^{2k}t), \tilde W(X, t)= r^{2ks} W(r^k X, r^{2k} t), \tilde W_1(X,t)= r^{k(2s - \alpha)} W_1(r^kX, r^{2k} t). 
 \]
 By change of variable it is seen from \eqref{sm5} that the following holds
 \[
 \int_{\mathbb B_1^+ \times (-1, 0]} |\tilde U|^2 y^a\leq 1.
 \]

 Furthermore the smallness condition in \eqref{sm1} is seen to be verified since $r< 1$ and $\alpha$ is chosen as in \eqref{kj}. Therefore by \emph{Step 2}, there exists   $\tilde b$, such that
 \begin{equation}\label{cls1}
 \int_{\mathbb B_r^+ \times (-r^2, 0]} |\tilde U - \tilde b|^2 y^a \leq r^{n+3+a+2\alpha}. 
 \end{equation}
 By rescaling back to $\bU$, we find that \eqref{sm5} follows from \eqref{cls1} with $b_{k+1} = b + r^{k \alpha} \tilde b$.

 \emph{Step 4:} (Conclusion) We first note that by rescaling as follows
 \[
 \bU_{r_0}(X,t) = \bU_{r_0}(r_0X, r_0^2 t)
 \]
 we find that $\bU_{r_0}$ solves \eqref{res1} with $B_{r_0}(X,t)= r_0 B(r_0X, r_0^2 t)$, $W_{r_0}(X, t)= r_0^{2s} W(r_0X, r_0^2 t)$ and $(W_1)_{r_0}(X,t)=r_0^{2s} W_1(r_0X, r_0^2t)$. Thus by choosing $r_0$ small enough, we can ensure that the smallness condition in \eqref{sm1} can be ensured.

 Now  from \eqref{sm5}, it follows by a standard real analysis argument that the following estimate holds for all $r < 1/2$ with $b_0= \text{lim}_{k \to \infty} b_k$
 \begin{equation}\label{cls2}
  \int_{\mathbb B_r^+ \times (-r^2, 0]} |\bU - b_0|^2  y^a \leq C r^{n+3+a+2\alpha}. 
  \end{equation}
    
  Similarly we have by translation in the tangential directions that  for every $(x_0,0, t_0) \in B_{1/2} \times \{0\} \times (-1/4, 0]$,  there exists $b(x_0, t_0)$ such that
  \begin{equation}\label{cls4}
  \int_{\mathbb B_r^+((x_0, 0)  \times (t_0-r^2, t_0]} |\bU - b(x_0, t_0)|^2  y^a \leq C r^{n+3+a+2\alpha}. 
  \end{equation}  
  Moreover the fact that  assignment $(x_0, t_0) \to  b(x_0, t_0)$ is $H^{\alpha}$ regular also follows in a standard way from \eqref{cls4}.  Now given the boundary decay estimate in \eqref{cls4}, combined with the fact that $\bU$ solves a uniformly parabolic PDE away from $y>0$, one can combine \eqref{cls4} with interior estimates  ( see for instance the proof of Theorem 2.3 in \cite{BBG}) to assert that $\bU$ is in $H^{\alpha}(\overline{\mathbb B_{1/2}^+} \times (-1/4, 0])$. We nevertheless provide the details for the sake of completeness.
  
  
  Let now $(X_1, t_1)$ and $(X_2, t_2)$ be two points in $\mathbb B_{1/2}^+ \times (-1/4, 0]$.  Our objective is to show that the following estimate holds
  \begin{equation}\label{des1}
  |\bU(X_1, t_1) - \bU(X_2, t_2)| \leq C |(X_1, t_1) - X_2, t_2)|^{\alpha}.
  \end{equation}

  Without loss of generality, we assume that $y_1 \leq y_2$. There are two cases:
  \begin{enumerate}
  \item $|(X_1, t_1) - (X_2, t_2)| \leq \frac{1}{4} y_1$;
  \item $|(X_1, t_1) - (X_2, t_2)| \geq \frac{1}{4} y_1$.  
  \end{enumerate}
  If $(1)$ occurs, then applying \eqref{cls4} with $r= \frac{y_1}{2}$, it ensues that the following $L^{2}$ bound is satisfied by
  $U^0(X,t)= \bU(X,t) - b(x_1, t_1)$
  \begin{equation}\label{try}
  \int_{\mathbb B_{\frac{y_1}{2}} (X_1) \times (t_1-\frac{y_1^2}{4},  t_1]} (U^0)^2 y^a \leq C (y_1)^{n+3+a+2\alpha}.
  \end{equation}
  We then observe that the rescaled function \begin{equation}\label{resc1}\tilde U_0(X, t)=  U_0( x_1 + y_1 x, y_1y, t_1 +y_1^2t)\end{equation} satisfies in $\mathbb B_{1/2}((0,1)) \times (-1/4, 0]$ a uniformly parabolic system with bounded  drift. Thus from the classical parabolic theory using \eqref{try} one has 
  \begin{align}\label{kest}
 & |\tilde U^0(X, t) - \tilde U^0(0,1, 0)| \leq C\left( \int_{\mathbb B_{1/2}((0,1)) \times (-1/4, 0] }  (\tilde U^0)^2 \right)^{1/2}  |(X,t) - (0,1, 0)|^{\alpha}\\
 & \leq \frac{C}{y_1^{\frac{n+3}{2}} } \left(  \int_{\mathbb B_{\frac{y_1}{2}} (X_1) \times (t_1-\frac{y_1^2}{4},  t_1]} (U^0)^2 \right)^{1/2} |(X,t) - (0,1,0)|^{\alpha}\ \text{(change of variable)}\notag\\
 & \leq  \frac{C}{y_1^{\frac{n+3+a}{2}} } \left(  \int_{\mathbb B_{\frac{y_1}{2}} (X_1) \times (t_1-\frac{y_1^2}{4},  t_1]} (U^0)^2 y^a\right)^{1/2} |(X,t) - (0,1,0)|^{\alpha}\notag
 \\
 & \leq Cy_1^{\alpha} |(X,t) - (0,1,0)|^{\alpha}.\notag
  \end{align}
In the second inequality in \eqref{kest} above, we used that $y \sim y_1$ in $\mathbb B_{y_1/2}(X_1)$ and in the last inequality we used the decay estimate \eqref{try}. The estimate  \eqref{des1}  follows in this case by applying the inequality \eqref{kest} with $(X,t)= (\frac{x_1 - x_2}{y_1}, \frac{y_2}{y_1}, \frac{t_2 - t_1}{y_1^2})$ and 
 by scaling back to $\bU$.
 
 If instead $(2)$ occurs, then we  first note that for $i=1,2$,  $\bU_i = \bU - b(x_i, t_i)$  solves the PDE \eqref{res1} in $\mathbb B_{y_i/2} \times (t_i - \frac{y_i^2}{4}, t_i]$.  After rescaling as above, using \eqref{try}( which also holds for $\bU_2$ with $y_1$ replaced by $y_2$), from the classical  estimates, we  obtain the following inequality
 \begin{equation}\label{im1}
 |\bU(X_i,t_i) - b(x_i, t_i)| \leq y_i^{\alpha}.
 \end{equation}
 Moreover, the triangle inequality gives
 \begin{equation}\label{im2}
 |y_2| = |X_2 - x_2| \leq |X_2 - X_1| + |X_1 - x_1| + |x_1 - x_2| \leq 6|X_1 - X_2|.
 \end{equation}
 Thus from \eqref{im1} and \eqref{im2} we have
 \[
 |\bU(X_1, t_1) - \bU(X_2, t_2)| \leq \sum_{i=1}^2  |\bU(X_i,t_i) - b(x_i, t_i)| + |b(x_1, t_1) - b(x_2, t_2)| \leq C |(X_1, t_1) - (X_2, t_2)|^{\alpha}.
 \]
 This finishes the proof of the theorem.

\end{proof}

Now following the H\"older continuity result in Theorem \ref{hold} which is the analogue of Theorem 5.1 in \cite{BG}, given such an $\alpha$,  starting with $k=1, 2,... [\frac{1}{\alpha}] +1$,  we  now take iterated difference quotients of  $\bU^*$ in $x$ and $t$ of the type

\begin{equation}\label{ty1}
\bU^*_{h, i}= \frac{\bU^*(x+he_i, y, t) - \bU^*(x, y,t)}{h^{k\alpha}},\ \bU^*_{h, t}= \frac{\bU^*(x+, y, t+h) - \bU^*(x, y,t)}{h^{\frac{k\alpha}{2}}},
\end{equation}
for   $i=1,.., n$
 similar to  in the proof of Lemma 5.6 in \cite{BG}. Note that $\bU^*_{h, i}$ and $\bU^*_{h, t}$ solves a problem of the type \eqref{res1} thanks to \eqref{vassump}, thus by applying Theorem \ref{hold} to these difference quotients we can conclude that $\nabla_x \bU^*, \bU^*_t$ are in $H^{\alpha}$ up to $\{y=0\}$.   Now by applying the method of difference quotients as in the proof of Theorem \ref{w22}, we can assert that $\nabla \nabla_x \bU^* \in L^{2}(\mathbb B_r^+ \times (-r^2, 0], y^a dXdt)$ for all $r <1 $. Moreover  since we have that $\bU^*_t \in H^{\alpha}$ up to $\{y=0\}$, it now follows by using the equation \eqref{staru} that $ (y^a \bU^*_y)_y  \in L^{2} (\mathbb B_r^+ \times (-r^2, 0], y^{-a} dXdt)$ for all $r<1$.  We record all such observations in the following lemma.
 
  \begin{lemma}\label{regular1}
Let $\bU^*$ be as in Lemma \ref{red1}. Then we have that the  $\nabla_x \bU^*, \partial_t \bU^* \in H^{\alpha}$ up to $\{y=0\}$ for all $\alpha$ as in Theorem \ref{hold}. Moreover, the following estimates hold
\begin{equation}\label{ws002}
||\bU^*||_{L^{\infty}( \mathbb B_{1/2}^+ \times (-1/4, 0])} + ||\nabla_x \bU^*, \partial_t \bU^*||_{H^{\alpha}(\mathbb B_{1/2}^+ \times (-1/4, 0])}\leq C \left(\int_{\mathbb B_1^+ \times (-1, 0]} |\bU^*|^2y^a\right)^{1/2} \end{equation} 
and
\begin{equation}\label{ws22}
 \int_{\mathbb B_{1/2}^+ \times (-1/4, 0]} y^a |\nabla \bU^*|^2 + y^a |\nabla \nabla_x \bU^*|^2  + y^a |\partial_t \bU^*|^2 + y^{-a} |(y^a \bU^*_y)_y |^2 \leq C \int_{\mathbb B_1^+ \times (-1, 0]} |\bU^*|^2y^a.\end{equation}\end{lemma}

We now show the H\"older continuity of $y^a \bU^*_y$ up to the thin set $\{y=0\}$.  

\begin{lemma}\label{holdery}
Let $\bU^*$ be as in Lemma \ref{red1}. Then we have that $y^a \bU^*_y \in H^{\beta}(\overline{\mathbb B_{1/2}^+} \times (-1/4, 0])$ for some $\beta>0$. Moreover the following estimate holds
\begin{equation}\label{ws000}
 ||y^a \bU^*_y||_{H^{\beta}(\mathbb B_{1/2}^+ \times (-1/4, 0])}\leq C \left(\int_{\mathbb B_1^+ \times (-1, 0]} |\bU^*|^2y^a\right)^{1/2}.\end{equation}

\end{lemma}
\begin{proof}
We first note that $\bU^*$ solves

\begin{equation}\label{staru1}
\begin{cases}
  y^a\dd_t {\bU}^* - \Div \left(y^a \nabla_{x,y}\bU^* \right) - y^a B \nabla {\bU^*} = 0\ \text{in $  \mathbb B_1^+ \times (-1, 0]$},
  \\   
   \py \bU^*=  \tilde V \bU^* = \phi \in H^{\alpha_0}\  \text{in $B_1 \times (-1, 0]$  for $\alpha_0=1$ (thanks to Lemma \ref{regular1})}.
  \end{cases}
\end{equation}
Now by an analogous compactness argument as in the proof of Lemma 5.6 in \cite{BS}, it follows that there exists $\alpha>0$ such that $1<\alpha +2s<2$ for which the following estimate holds
\begin{equation}\label{decay2}
\int_{\mathbb B_r^+((x_0, 0)) \times (t_0- r^2, t_0])} \left( \bU^* - \frac{1}{1-a} \phi(x_0, t_0) y^{1-a} - <\nabla_x \bU^*(x_0, t_0), x-x_0> \right) ^2 y^a \leq Cr^{n+3+a +2(\alpha + 2s)}.
\end{equation}
Given $(x_0,t_0) \in B_{1/2} \times (-1/4, 0]$, 
Let $H= y^a \bU^*_y - \phi(x_0, t_0)$.  Then from the action of the drift $B$ as in  \eqref{drift1} and Lemma \ref{regular1}  it follows from a direction calculation that $H$ is a weak solution of the  following conjugate system 
\begin{equation}\label{kj02}
  y^{-a}\dd_t H - \Div \left(y^{-a} \nabla_{x,y}H \right) - y^{-a} B \nabla {H} = 0.\end{equation}
  Now for a point $(X_0, t_0)=(x_0, y_0, t_0)$, let $H^*=  \bU^* - \frac{1}{1-a} \phi(x_0, t_0) y^{1-a} - <\nabla_x \bU^*(x_0, t_0), x-x_0> $. Then we note that $H^*$   solves the system of the type
  \begin{equation}\label{kjo4}
    y^a\dd_t H^* - \Div \left(y^a \nabla_{x,y}H^* \right) - y^a B \nabla {H^*} = y^a F
    \end{equation}
    in $\mathbb B_{\frac{y_0}{2}}(X_0) \times (t_0- y_0^2/4, t_0]$. Then from the energy estimate applied to $H^*$   it follows that the following inequality holds
  \begin{align}\label{kjo1}
&  \int_{\mathbb B_{\frac{y_0}{2}}(X_0) \times (t_0- y_0^2/4, t_0]} |H|^2 y^{-a} \leq  \int_{\mathbb B_{\frac{y_0}{2}}(X_0) \times (t_0- y_0^2/4, t_0]} |\nabla H^*|^2 y^{a}\\& \leq \frac{C}{y_0^2}  \int_{\mathbb B_{\frac{3y_0}{4}}(X_0) \times (t_0- 9y_0^2/16, t_0]} \left( |H^*|^2 + y_0^4 |F|^2\right) y^a.\notag
 \end{align}
 Now using the decay estimate \eqref{decay2}, we obtain the following bound from \eqref{kjo1}
 \begin{equation}\label{bdd1}
  \int_{\mathbb B_{\frac{y_0}{2}}(X_0) \times (t_0- y_0^2/4, t_0]} |H|^2 y^{-a}\leq Cy_0^{n+1+a+2(\alpha+2s)}.
  \end{equation}
  Now since $H$ solves \eqref{kj02},  by rescaling as in \eqref{resc1} we note that the rescaled function solves a uniformly parabolic system in  $\mathbb B_{1/2}((0,1)) \times (-1/4, 0]$ and thus by applying the classical estimates to the rescaled function and by scaling back we obtain the following bound
  \begin{equation}\label{bdd2}
  |H(X_0,t_0)|= |y^a \bU^*_y(X_0, t_0) - \phi(x_0, t_0)| \leq Cy_0^{\alpha}.
  \end{equation}
  
  Again let $(X_1, t_1)$ and $(X_2, t_2)$ be two points in $\mathbb B_{1/2}^+ \times (-1/4, 0]$.  Our objective is to show that the following estimate holds
  \begin{equation}\label{des1}
  |y^a\bU^*_y(X_1, t_1) - y^a\bU^*_y(X_2, t_2)| \leq C |(X_1, t_1) - X_2, t_2)|^{\alpha}.
  \end{equation}

  Without loss of generality, we assume that $y_1 \leq y_2$. There are two cases:
  \begin{enumerate}
  \item $|(X_1, t_1) - (X_2, t_2)| \leq \frac{1}{4} y_1$;
  \item $|(X_1, t_1) - (X_2, t_2)| \geq \frac{1}{4} y_1$.  
  \end{enumerate}
  
  If (1) occurs, then the function $H_1= y^a \bU^*_y - \phi(x_1, t_1)$ satisfies an equation of the type \eqref{kj02} in $\mathbb B_{\frac{y_1}{4}}(X_1) \times (t_1 - y_1^2/4, t_1]$. Again by rescaling as in \eqref{resc1}, we note that the rescaled function satisfies a uniformly parabolic system  in   $\mathbb B_{1/2}((0,1)) \times (-1/4, 0]$. Arguing as in \eqref{try}-\eqref{kest} we thus obtain
  \begin{align}\label{bn1}
 & |y^a\bU^*_y(X_1, t_1) - y^a\bU^*_y(X_2, t_2)| = |H_1(X_1, t_1) - H_1(X_2, t_2)| 
 \\
 & \leq \frac{C}{y_1^{\alpha}} \left( \frac{1}{y_1^{n+3-a}} \int_{\mathbb B_{y_1/2}(X_1) \times (t_1 - y_1^2/4, t_1]} H^2 y^{-a} \right)^{1/2} |(X_1, t_1)- (X_2, t_2)|^{\alpha}
 \notag
 \\
 & \leq C|(X_1,t_1)-(X_2, t_2)|^{\alpha}.\notag
 \end{align}
 In the first inequality in \eqref{bn1} we used that in $\mathbb B_{y_1/2}(X_1)$, $y \sim y_1$ and in the last inequality, we used the decay estimate in \eqref{bdd1} with $y_0$ replaced by $y_1$. Suppose now (2) occurs.    In this case, we note that \eqref{bdd2} holds with $(X_0, t_0)$ replaced by $(X_1, t_1)$ and $(X_2, t_2)$. More precisely we have the following inequality
 \begin{equation}\label{bdo1}
 |y^a \bU^*_y (X_i, t_i) - \phi(x_i, t_i)| \leq Cy_i^{\alpha} 
 \end{equation}
 for $i=1,2$. Moreover from \eqref{im1} we have that $|y_2| \leq 6|(X_1,t_1) - (X_2, t_2)|.$. Consequently, we obtain
 \begin{align}\label{bdo2}
& |y^a \bU^*_y(X_1, t_1) - y^a \bU^*_y (X_2, t_2)| \leq \sum_{i=1}^2 |y^a \bU^*_y (X_i, t_i) - \phi(x_i, t_i)|  + |\phi(x_1, t_1) - \phi(x_2, t_2)|\\ &\leq Cy_1^{\alpha} +Cy_2^{\alpha} + C|(x_1, t_1) - (x_2, t_2)|^{\alpha} \leq C|(X_1,t_1)-(X_2, t_2)|^{\alpha}. \notag\end{align}

This finishes the proof of the  lemma.

\end{proof}

Finally by combining the energy estimate in Theorem \ref{w22} along with the regularity results in  Lemma \ref{regular1}, we can assert that the following regularity estimate holds for solutions to the original extension problem  \eqref{ext1} which will be needed subsequently in the proof of Theorem \ref{main}.

\begin{lemma}\label{regular2}
Let $\Tbu$ be as in \eqref{ext1}. Then the following estimates hold for $\alpha$ as in Lemma \ref{regular1}. 
\begin{align}\label{lat}
||\Tbu||_{L^{\infty}( \mathbb B_{1/2}^+ \times (-1/4, 0])} + ||\nabla_x \Tbu, \partial_t \Tbu||_{H^{\alpha}(\mathbb B_{1/2}^+ \times (-1/4, 0])}\leq C \left(\int_{\mathbb B_1^+ \times (-1, 0]} |\Tbu|^2y^a\right)^{1/2} 
\end{align}
and
\begin{equation}\label{lat1}
 \int_{\mathbb B_{1/2}^+ \times (-1/4, 0]} y^a |\nabla \Tbu|^2 + y^a |\nabla \nabla_x \Tbu|^2  + y^a |\partial_t \Tbu|^2 + y^{-a} |(y^a \Tbu_y)_y |^2 \leq C \int_{\mathbb B_1^+ \times (-1, 0]} |\Tbu|^2y^a.\end{equation}
\end{lemma}

\section{ Proof of Theorem \ref{main}}\label{s:mn}
As previously mentioned  in the introduction, it turns out that some subtle obstructions  which will be described below  doesn't allow us to adapt the Poon type frequency approach    to the reduced system \eqref{staru}. Such an obstruction is very specific to the vectorial case. However with  the smoothness assumptions on the potential $V$ as in \eqref{vassump}, it turns out that in the range $s \geq 1/2$, a certain  change of variables can be carried out which reduces the system \eqref{staru} further to  another system with zero Neumann conditions.

Before proceeding further, we mention that for  notational convenience,  we will denote $\bU^*$ as in \eqref{star}   by $U$ throughout this section.  Similar to that in \cite{ABDG}, by letting $t \to -t$ and by rescaling, we may assume that $U$ solves the following backward parabolic problem

\begin{equation}\label{extlame1}
\begin{cases}
  y^a\dd_t U + \Div \left(y^a \nabla_{x,y}U \right) - y^a B \nabla {U} = 0\ \text{in $\mathbb Q_4^+ \overset{def} = \mathbb B_4^+ \times [0, 16)$},
  \\   
   \py U=  \tilde V U\ \text{in $B_{4} \times [0, 16)$}.
  \end{cases}
\end{equation}

Our first lemma is a monotonicity in time result which allows the passage of information to $t=0$ for the extension problem \eqref{extlame1}.  We now introduce an assumption that will remain in force till the proof of Theorem \ref{main}. We will assume that

\begin{equation}\label{asum}
\int_{\mathbb B_1^+} |U(\cdot, 0)|^2 y^a>0.
\end{equation}
As a consequence of such hypothesis the number
\begin{equation}\label{theta}
\theta \overset{def}= \frac{\int_{\mathbb B_4^+ \times [0, 16)} |U|^2 y^a dXdt}{\int_{\mathbb B_1^+} |U(\cdot, 0)|^2 y^adX}
\end{equation}
is well defined.  We now state and prove the relevant monotonicity in time result.  Before proceeding further, we would like to mention throughout this section, for a vector valued function $f$, we will denote its components by $f^i$.
\begin{lemma} \label{mont}
Let $U$ be a solution of \eqref{extlame1} in $\mathbb Q_4^+.$ Then there exists a constant $N=N(n,a, |B|, ||\tilde V||_{C^1})>2$ such that $N \operatorname{log}(N\theta) \geq 1$, and for which the following inequality holds for $0 \leq t\leq \frac{1}{N \log(N\theta)}$
\[
N\int_{\mathbb{B}_2^+} |U(x,t)|^2 y^a dX\geq \int_{\mathbb{B}_1^+} |U(x,0)|^2 y^a dX.
\]
\end{lemma}
\begin{proof}
The proof is similar to that in \cite{ABDG}.
Let  $f= \phi U,$ where $\phi \in C_0^{\infty}(\mathbb B_2)$ is a spherically symmetric cutoff such that $0\le \phi\le 1$ and $\phi \equiv1$ on $\mathbb B_{3/2}.$ Since $U$ solves \eqref{extlame1} and $\phi$ is independent of $t$, it is easily seen that the function $f$ solves the problem
\begin{align} 
\begin{cases}\label{extended parabolic Lame2}
  y^a \dd_tf +\Div(y^a \nabla_{x,y}f) -y^a{B}\nabla f= 2y^a\langle \nabla \phi, \nabla U \rangle +\text{div} (y^a \nabla \phi) U -y^a {B} \nabla \phi U, \
    \tn{ for } (x,t)\in \mathbb Q_4^+,\\
    \py f= \tilde V(x,t) \, f(x,t)  \text{ in } B_4 \times [0,16).
\end{cases}    
\end{align}
As in \cite{ABDG}, $\phi$ has the following properties:
\begin{equation}\label{obs1}
\begin{cases}
\operatorname{supp} (\nabla \phi) \cap \{y>0\}  \subset \mathbb B_2^+ \setminus \mathbb B_{3/2}^+
\\
|\Div(y^a \nabla \phi)| \leq C y^a\ \mathbf 1_{\mathbb B_2^+ \setminus \mathbb B_{3/2}^+},
\end{cases}
\end{equation}
where for a set $E$ we have denoted by $\mathbf 1_E$ its indicator function.

Let us fix a point $X_1=(x_1,y_1)\in \R^{n+1}_+$ and introduce the quantity
\begin{align*}
H(t) = \int_{\R^{n+1}_+} |f(X,t)|^2 \G(X_1,X,t) y^a dX 
\end{align*}
where 
\begin{equation}\label{funda}
\G(X_1,X,t) = p(x_1,x,t)\ p^{(a)}(y_1,y,t)
\end{equation}
is the product of the standard Gauss-Weierstrass kernel $p(x_1,x,t) = (4\pi t)^{-\frac n2} e^{-\frac{|x_1-x|^2}{4t}}$ in $\Rn\times\R^+$ with the heat kernel of the Bessel operator $\mathscr B_a = \p_{yy} + \frac ay \p_y$  with Neumann boundary condition in $y=0$ on $(\R^+,y^a dy)$ (reflected Brownian motion)
 \begin{align}\label{fs}
p^{(a)}(y_1,y,t) & =(2t)^{-\frac{a+1}{2}}\left(\frac{y_1 y}{2t}\right)^{\frac{1-a}{2}}I_{\frac{a-1}{2}}\left(\frac{y_1 y}{2t}\right)e^{-\frac{y_1^2+y^2}{4t}}.
\end{align}

\par

From the asymptotic behaviour of $I_{\frac{a-1}{2}}(z)$ near $z=0$ and at infinity one immediately obtains the following estimate for some $C(a), c(a) >0$ (see e.g. \cite[formulas (5.7.1) and (5.11.8)]{Le}) ,
\begin{equation}\label{bessel}
I_{\frac{a-1}{2}}(z) \leq C(a) z^{\frac{a-1}{2}}  \hspace{6mm} \text{if} \hspace{2mm} 0 < z \le c(a),\ \ \ \ \ I_{\frac{a-1}{2}}(z) \leq C(a) z^{-1/2} e^z \hspace{2mm}\ \  \text{if} \hspace{2mm} z \ge c(a).
\end{equation}
Note that as $z\to 0^+$
\begin{equation}\label{zero}
z^{\frac{1-a}{2}} I_{\frac{a-1}{2}}(z) \cong \frac{2^{\frac{1-a}{2}}}{\Gamma((1+a)/2)}.
\end{equation}

\par

Observe that \eqref{funda} and \eqref{fs} imply that for every $x, x_1\in \Rn$ and $t>0$ one has
\begin{equation}\label{neu}
\operatorname{lim}_{y \to 0^+}  y^a \partial_y \G((x,y),(x_1,0),t) =0. 
\end{equation}

We observe explicitly that by the semigroup property of $\mathscr P_t^{(a)}$ and the fact that $\phi \equiv 1$ in $\mathbb B_1$, we have for every $X_1\in \mathbb B_1^+$
\begin{equation}\label{H0}
\underset{t\to 0^+}{\lim}\ H(t) = |f(X_1,0)|^2 = |U(X_1,0)|^2. 
\end{equation}
We want to establish the following.

\noindent \emph{\underline{Claim}}: There exist constants $C=C(n,a, |B|_{\infty}, ||\tilde V||_1)$ and $0<t_0 = t_0(n,a,||\tilde V||_{\infty})<1$ such that for $X_1\in \mathbb B_1^+$ and $0<t<t_0$ one has 
\begin{equation}\label{mkg1}
H'(t) \geq -C||\tilde V||_{\infty} t^{-\frac{1+a}2} H(t) -C e^{-\frac{1}{Ct}}||U||_{L^2(\mathbb Q_4^+, y^a dXdt)}^2.
\end{equation}

\medskip

Using the divergence theorem, \eqref{extended parabolic Lame2} and the Neumann condition \eqref{neu} above we obtain for any fixed $X_1\in \mathbb B_1$ and $0<t\le 1$
\begin{align}\label{h1}
H'(t) &  = \sum_i\int 2 f^i f^i_t \G y^a dX +\int {f^i}^2 \G_t y^a dX
=\sum_i \int 2 f^i f^i_t \G y^a dX +\int {f^i}^2 \operatorname{div} (y^a  \nabla\G)dX
\\
&= \sum_i\int 2 f^i f^i_t \G y^a dX -\int \langle\nabla {f^i}^2,\nabla \G\rangle y^a dX
= \sum_i\int 2 f^i f^i_t \G y^a dX +\int \operatorname{div}(y^a \nabla {f^i}^2) \G dX
\notag\\
& +\int_{\{y=0\}} 2\tilde Vf \cdot f \, \G dx
\notag \\
&=\sum_i \int 2 f^i (y^a f^i_t + \operatorname{div}(y^a \nabla f^i)) \G  dX
 + \int 2 |\nabla f^i|^2   \G y^a dX +\int_{\{y=0\}} 2\tilde Vf \cdot f \, \G dx\notag\\
&= I_1 + I_2 + I_3.\notag
\end{align}

\par 

We first prove that
\begin{itemize}
\item[(i)] for every $X_1\in \mathbb B_1^+$, $0<t\le 1$ and $\epsilon>0$, there exists $C=C(n,a, |B|_{\infty})$ such that the following inequality holds
\begin{equation}\label{i1}
I_1 \geq - C e^{-\frac{1}{M t}} \int_{\mathbb Q_4^+} |U|^2 y^a dXdt+\left[\sum_i[C \, \epsilon \int  |\nabla f^i|^2    y^a \G \, dX+\frac{C}{\epsilon} \int  {f^i}^2y^a  \G \,  dX]\right]. 
\end{equation}

\par

\item[(ii)] there exists $t_0<1$ such that for every $X_1\in \mathbb B_1^+$ and $0<t\le t_0$ one has
\begin{align}\label{traceapp}
|I_3| \leq C(n,a)||\tilde V||_{\infty}\sum_i\left(t^{-\frac{1+a}{2}}\int {f^i}^2 \G y^a dX + t^{\frac{1-a}{2}} \int |\nabla f^i|^2 \G y^a dX\right).
\end{align}
\end{itemize}

\par

With \eqref{i1} and \eqref{traceapp} in hands, we return to \eqref{h1} to find
\begin{align*}
H'(t)  \ge & - C e^{-\frac{1}{M t}} \int_{\mathbb Q_4} U^2 y^a dXdt + \sum_i2\int  |\nabla f^i|^2   \G y^a dX
\\
& - C(n,a)||\tilde V||_{\infty}\left(t^{-\frac{1+a}{2}} H(t) + t^{\frac{1-a}{2}} \sum_i\int |\nabla f^i|^2 \G y^a dX\right)\\
&- \left[\sum_i[C \, \epsilon \int  |\nabla f^i|^2    y^a \G \, dX+\frac{C}{\epsilon} \int  {f^i}^2y^a  \G \,  dX]\right].\end{align*}

\par

If at this point in this inequality we choose $t_0<1$  and $\ve>0$  such that $C \ve + C(n,a)||\tilde V||_{\infty} t_0^{\frac{1-a}{2}} <1$, it is clear that  for $X_1\in \mathbb B_1^+$ and $0<t<t_0$ we obtain the \emph{Claim} \eqref{mkg1}.

\par

We further note that \eqref{traceapp} in fact follows from the inequality \cite[(3.14)]{ABDG}.  Therefore, it suffices to prove the estimate \eqref{i1}. Next when $X_1 \in \mathbb B_1^+$, $X \in \mathbb B_2^+ \setminus  \mathbb B_{3/2}^+$ and $0<t\le 1$, the following bound holds for some universal $M>0$ (See \cite[(3.16)]{ABDG})
\begin{equation}\label{g4}
\G(X_1, X, t) \leq e^{-\frac{1}{M t}}.
\end{equation}

\par

With this estimate in hand, we now insert such inequality in the definition of $I_1$ and using \eqref{extended parabolic Lame2} and \eqref{obs1} we finally obtain for $\epsilon>0$
\begin{align*}
&|I_1| \le C e^{-\frac{1}{Mt}} \sum_i \int_{\mathbb B_2^+} \left(|\nabla U^i| + |U^i|\right) |U^i| y^a+ C\sum_i\int | f^i (y^a[{B}\nabla f]^i)| \G  dX\\
&\leq C e^{-\frac{1}{Mt}} \, \sum_i \int_{\mathbb B_2^+} \left(|\nabla U^i| + |U^i|\right) |U^i| y^a+ \sum_i[C \, \epsilon \int  |\nabla f^i|^2    y^a \, \G \, dX+\frac{C}{\epsilon} \int  {f^i}^2y^a   \G \, dX],
\end{align*}
by Cauchy Schwarz inequality. Here $C$ depends on $|B|_{\infty}.$
At this point we invoke the $L^{\infty}$ bounds for $U, \nabla_x U^i, U_t$ and $y^a U_y$ in Lemma \ref{regular1} and Lemma \ref{holdery} to finally conclude that for every $X_1\in \mathbb B_1^+$ and $0<t\le 1$ the inequality \eqref{i1} holds and this completes the proof of \eqref{mkg1}.

Once the \emph{\underline{Claim}} \eqref{mkg1} is proved  which is the analogue of \cite[(3.7)]{ABDG}, we can now repeat the arguments  as in \cite{ABDG} using the regularity estimate in Lemma \ref{regular1} to complete the proof of the lemma. 

\end{proof}

\subsection{A further reduction of \eqref{extlame1}}\label{furred}

Now it turns out that due to the asymmetric nature of the matrix potential $\tilde V$ which has a similar structure as in \eqref{matr}, it is not obvious to adapt  the Poon type approach as in  the scalar case  in \cite{ABDG} to obtain the analogues of the   first variation results   in \cite[Lemma 3.3]{ABDG}.  More precisely, due to the structure of the matrix potential in \eqref{matr}, one cannot ensure certain key cancellations in the proof of  the variation of the energy in Lemma 3.3 in \cite{ABDG}  which is one of the  important ingredients in the proof of the doubling inequality.   However in the case when $s \geq 1/2$, with a higher  $C^{2}$ regularity assumption on $\tilde V$ which in turn is guaranteed by \eqref{vassump}, we show that by a certain transformation, we can get rid of the Neumann datum.  This transformation also exploits very crucially the special structure of $\tilde V$ as in \eqref{matr}. Over here it is to be said that althought  $\tilde V$ is not symmetric, nevertheless it has a specific structure that ensures that $\tilde V$ commutes with $\partial_t \tilde V, \nabla_x \tilde V$  which in turn allows for some of the ensuing computations below.  As previously mentioned in the introduction, this feature is very specific to the vectorial case.

Corresponding to $U$ as in \eqref{extlame1}, we let

\begin{equation}\label{W}
W= e^{-\frac{\tilde V(x,t) y^{1-a}}{1-a}} U.
\end{equation}
Therefore we have for $i\in\{1,2,\cdots,n\}$
\begin{align*}
	\dd_t W = e^{-\f{y^{1-a}}{1-a} \tilde{V}(x,t)} \left(\dd_t U - \f{y^{1-a}}{1-a} V_t U\right), \qd
	\dd_{x_i} W = e^{-\f{y^{1-a}}{1-a} \tilde{V}(x,t)} \left(\dd_{x_i} U - \f{y^{1-a}}{1-a} V_{x_i} U\right)
\end{align*}
and
\begin{equation}
y^a \dd_y W = e^{-\frac{\tilde V(x,t) y^{1-a}}{1-a}} (- \tilde V) U + e^{-\frac{\tilde V(x,t) y^{1-a}}{1-a}}  y^a U_y.
\end{equation}
 By letting $y \to 0^+$,  it  thus follows from the Neumann condition in \eqref{extlame1} that
\begin{equation}\label{zeron}
\py W=0.
\end{equation}
After a further computation, we find that
\begin{align*}
	& \dd^2_{x_ix_i} W = e^{-\f{y^{1-a}}{1-a}\tilde{V}(x,t)} \left( \dd^2_{x_ix_i}U  - \f{2y^{1-a}\tilde{V}_{x_i}\dd_{x_i}U}{1-a} \right) +  \left(  \tilde{V}^2_{x_i} \f{y^{2(1-a)}}{(1-a)^2} W  - \tilde{V}_{x_ix_i} \f{y^{1-a}W}{1-a} \right), \\
	& \partial_y (y^a W_y) = - \tilde V W_y + e^{-\f{y^{1-a}}{1-a} \tilde V(x,t)} \partial_y (y^a U_y) - \tilde Ve^{-\f{y^{1-a}}{1-a} \tilde V(x,t)}U_y.
	\end{align*}
An important constituent in the above derivations lies in the observation that $\tilde{V}$ and $\tilde{V}_t$ (or $\tilde{V}_{x_i}$, $\tilde{V}_{x_i x_i}$) commute and thus we can compute the derivatives of exponential matrices.  This crucial aspect is easily verified from the structure of $\tilde V$ as in \eqref{matr}.

 In view of the above relations, we obtain  using the equation \eqref{extlame1} satisfied by $U$ that the following holds

 \begin{align}\label{calc1}
& y^a W_t + \Div(y^a \nabla W)= e^{-\f{y^{1-a}}{1-a} \tilde V(x,t)} y^aB \nabla U  - \left( \frac{\tilde V_t y}{1-a}  + \frac{y}{1-a} \Delta V\right)W\\
& + \f{y^{2-a}}{(1-a)^2} \sum_{i} \tilde{V}^2_{x_i} W    -\f{2y}{1-a} \sum_{i=1}^{n}  \tilde{V}_{x_i} e^{-\f{y^{1-a}}{1-a}\tilde{V}(x,t)} \dd_{x_i} U\notag\\
& - \tilde V W_y - \tilde Ve^{-\f{y^{1-a}}{1-a} \tilde V(x,t)}U_y.\notag 
\end{align}
Now again from \eqref{W} we have
\begin{align*}
	\dd_{x_i} U = e^{\f{y^{1-a}}{1-a}\tilde{V}(x,t)} \left(\dd_{x_i}W + \f{y^{1-a} \tilde{V}_{x_i}}{1-a} W\right),\ \dd_y U = e^{\f{y^{1-a}}{1-a}\tilde{V}(x,t)} \left(\dd_y W + y^{-a}\tilde{V}W\right).
\end{align*}
Using this in \eqref{calc1} above, we get
 
 \begin{align}\label{calc2}
& y^a W_t + \Div(y^a \nabla W)= y^ae^{-\f{y^{1-a}}{1-a} \tilde V(x,t)} Be^{\f{y^{1-a}}{1-a}\tilde{V}(x,t)} \nabla W  - \left( \frac{\tilde V_t y}{1-a}  + \frac{y}{1-a} \Delta \tilde V\right)W\\
& - \f{y^{2-a}}{(1-a)^2} \sum_{i} \tilde{V}^2_{x_i} W   -\f{2y}{1-a} \sum_{i=1}^{n}  \tilde{V}_{x_i}  \dd_{x_i} W\notag\\
& - 2\tilde V W_y - y^{-a}(\tilde V)^2W + y e^{-\f{y^{1-a}}{1-a} \tilde V(x,t)} Be^{\f{y^{1-a}}{1-a}\tilde{V}(x,t)}\frac{\nabla_x \tilde V}{1-a} W
.\notag 
\end{align}
 Over here  the term  $Be^{\f{y^{1-a}}{1-a}\tilde{V}(x,t)} \nabla W$ is to be interpreted as 
\begin{align*}
	Be^{\f{y^{1-a}}{1-a}\tilde{V}(x,t)} \nabla W := B(e^{\f{y^{1-a}}{1-a}\tilde{V}(x,t)} \nabla W).
\end{align*}
Note that $B$ acts as a linear map  from $\R^{(n+1) \times (n+1)}  \to \R^{n+1}$.
Similarly the following   term  in \eqref{calc2} above, i.e. $$Be^{\f{y^{1-a}}{1-a}\tilde{V}(x,t)}\frac{\nabla_x \tilde V}{1-a} W,$$ is  to be understood as 
\begin{align*}
	Be^{\f{y^{1-a}}{1-a}\tilde{V}(x,t)}\frac{\nabla_x \tilde V}{1-a} W := \f{1}{1-a} B \begin{pmatrix} 
		 e^{\f{y^{1-a}}{1-a}\tilde{V}(x,t)} \tilde V_{x_1}W & \cdots & e^{\f{y^{1-a}}{1-a}\tilde{V}(x,t)} \tilde V_{x_n}W, & 0 
	\end{pmatrix}.
\end{align*}

It thus  follows from \eqref{calc2} that when $a \leq 0$  (or equivalently when $s \geq 1/2$), then one has that  $W$  solves the following differential inequality $\B_4^+ \times [0, 16)$
\begin{equation}\label{weq1}
\left| y^a W_t + \Div(y^a \nabla W)  \right| \leq  C y^a (|W| + |\nabla W|),
\end{equation}
 with zero Neumann conditions.

Furthermore, we note that since $W$ and $U$ are related by the transformation as in \eqref{W} which is invertible in a bounded manner as $\tilde V$ is bounded, therefore we have that the assumption \eqref{asum} as well as the monotonicity result in Lemma \ref{mont} holds for $W.$
Moreover since the zero Neumann condition \eqref{zeron} holds, at this point we can apply the approach as  in the classical case  in \cite[Section 2]{EFV}  based on the Poon's frequency function ( more precisely we  use the  adaptation of such an  approach as  in \cite{ABDG} where $a \neq 0$ is considered) combined with the regularity estimates in Lemma \ref{regular1} and Lemma \ref{holdery}  to finally conclude that a conditional doubling property similar to that in \cite[Theorem 3.5]{ABDG} holds for $W$ and consequently for $U$ which can be stated as follows.

\begin{thrm}\label{db1}
Let $U$ be a solution of \eqref{extlame1} in $\mathbb Q_4^+$ and assume that $a\leq 0$.  There exists $N>2$, depending on $n$, $a$, sup norm of $B$ and the $C^2  $-norm of $\tilde V$, for which $N\log(N\theta) \ge 1$ and such that:
\begin{itemize}
\item[(i)] For $r \leq 1/2,$ we have 
$$\int_{\mathbb B_{2r}^+}|U(X,0)|^2 y^adX \leq (N \theta)^N\int_{\mathbb B_{r}^+}|U(X,0)|^2 y^adX.$$
Moreover for  $r \leq 1/\sqrt{N \operatorname{log}(N \theta)}$  the following two inequalities hold:
\item[(ii)] $$\int_{\mathbb Q_{2r}^+} |U(X,t)|^2y^a dXdt \leq \operatorname{exp}(N \operatorname{log}(N \theta) \operatorname{log}(N \operatorname{log}(N \theta)))r^2 \int_{\mathbb B_r^+}|U(X,0)|^2y^adX.$$
	 	\item[(iii)]$$\int_{\mathbb Q_{2r}^+} |U(X,t)|^2 y^adXdt \leq \operatorname{exp}(N \operatorname{log}(N \theta) \operatorname{log}(N \operatorname{log}(N \theta)))\int_{\mathbb Q_r^+}|U(X,t)|^2y^adXdt.$$ 
\end{itemize}
\end{thrm}
With the doubling property in Theorem \ref{db1} in hand, we now proceed with the proof of Theorem \ref{main}.

\begin{proof}[Proof of Theorem \ref{main}]The proof is divided into the following steps.

\emph{Step 1:} We first show that $|U|$ ( keeping in mind  that we are denoting $\bU^*$ in Lemma \ref{red1} by $U$) vanishes to infinite order at $(0,0)$ in the sense of \eqref{vp}. Now since $U$( i.e. $\bU^*$) is of the type
\begin{align*}
      \begin{pmatrix}
          \Tbu \\
          \tn{div}_x \Tbu
          \end{pmatrix},
\end{align*} 
and already by our assumption, $\Tbu$ vanishes to infinite order in $x$ and $t$ at $(0,0)$, therefore we only need to 
additionally show that $\Div_x \Tbu$ vanishes to infinite order in the tangential variables $x,t$ at $(0,0)$. It suffices to show that given $r>0$, 
\begin{equation}\label{vp1}
||\Div_x \Tbu||_{L^{2}(B_r \times [0, r^2))} = O(r^k),
\end{equation}
for all $k>0$.

 In order to show this, we use an idea based on an  interpolation type argument as in \cite{AB}.  See also \cite{Ru1} for the stationary case.

We first state the relevant interpolation inequality that we need which is Lemma 2.3 in \cite{AB}.
Given $s \in (0,1)$ and $f \in C^2_0(\R^n \times \R_+),$ there exists a universal constant $C$ such that for any $0<\eta <1$ the following holds
	\begin{align}\label{inte}
		||\nabla_x f||_{L^{2}(\R^n)} \leq C \eta^s\left( || y^{a/2} \nabla \nabla_x f||_{L^2(\R^n \times \R_+)} +|| y^{a/2} \nabla_x f||_{L^2(\R^n \times \R_+)}\right) + C\eta^{-1} ||f||_{L^2(\R^n)}.
	\end{align}
	
	Let $\phi$ be a smooth function of $X$ supported in ${\mathbb B_{4{r}}}$ such that $\phi \equiv 1$ in  
${\mathbb B_{r}}$.  We now apply the interpolation inequality \eqref{inte}  to $f=\phi \Tbu$ and  obtain for any $0<\eta <1$ that the following estimate holds
\begin{align}\label{der1}
	&||\n_x f||_{L^2(\R^{n} \times [0, r^2))}\\ &\le C\left( \eta^{s}(||y^{a/2}\n \n_x f||_{L^2(\R^{n} \times \R_+ \times [0,r^2))} +||y^{a/2}\n_xf||_{L^2(\R^{n} \times \R_+ \times [0, r^2))})
	 + \eta^{-1}||f||_{L^2(\R^n \times [0, r^2))} \right)\notag\\
	& \le C\big(\eta^{s}(|\phi|||y^{a/2}\n \n_x\Tbu||_{L^2(\mathbb Q_{2{r}}^+)} +(|\n\phi|+|\phi|)||y^{a/2}\n \Tbu||_{L^2(\mathbb Q_{2{r}}^+)}\notag\\
	& +(|\n^2 \phi|+|\n\phi|)||y^{a/2} \Tbu||_{L^2(\mathbb Q_{2{r}}^+)})+ \eta^{-1}||\Tbu||_{L^2( B_{2r} \times [0, 4r^2))} \big)\notag\\
	 & \le   \eta^{s}{r}^{-2}C ||y^{a/2} \Tbu||_{L^2(\mathbb Q_{4{r}}^+ )} +C \eta^{-1}||\Tbu||_{L^2( B_{2r} \times [0, 4r^2))},\notag
\end{align} 
where  in the last inequality, we used the rescaled versions of the regularity estimate in Lemma \ref{regular2}.  From \eqref{der1} it follows
\begin{align}\label{d1f}
	||\n_x \Tbu||_{L^2(B_r \times [0, r^2))} \le  C \eta^{s}{r}^{-2} ||y^{a/2} \Tbu||_{L^2(\mathbb Q_{4}^+ )} +C \eta^{-1}||\Tbu||_{L^2( B_{2r} \times [0, 4r^2)}.
\end{align}
Now since $\Tbu$ vanishes to infinite order in the sense of \eqref{vp}, we have that given $k \in \mathbb N$ large enough, there exists $C_{k,s}$ such that
\begin{equation}\label{d2f}
||\Tbu||_{L^2( B_{2r} \times [0, 4r^2)} \leq C_{k,s} r^{{2k/s}}.
\end{equation}
By letting $\eta= r^{k/s}$ and using \eqref{d2f} in \eqref{d1f} we find for a new $\tilde C_{k,s}$,
\begin{equation}\label{d4f}
||\n_x \Tbu||_{L^2(B_r \times [0, r^2))}  \leq  \tilde C_{k,s} r^{k-2}.
\end{equation}
Since $k$ is arbitrary, \eqref{vp1} follows.

\emph{Step 2:} From \emph{Step 1}, we obtain that $|U|$ vanishes to infinite order in the sense of \eqref{vp} at $(0,0)$ in the tangential  variables $x$ and $t$. We can now argue as in the proof of Theorem 1.1 in \cite{ABDG}.

We first show that  we must have
\begin{equation}\label{crucial}
U(X,0) \equiv 0, \ \ \ \ \ \ \ \text{for every}\ X\in \mathbb B_1^+.
\end{equation}
We argue by contradiction and assume that \eqref{crucial} is not true. Consequently, \eqref{asum} does hold. In particular, from \eqref{asum} and (i) in Theorem \ref{db1} it follows that $\int_{\mathbb B_r^+} |U(X, 0)|^2 y^a dX > 0$ for all $0<r \leq \frac{1}{2}$. From this fact and the continuity of $U$ up to the thin set $\{y=0\}$ we deduce that
\begin{equation}\label{nzero}
\int_{\mathbb Q_r^+} |U|^2 y^a dX dt>0, \end{equation} for all $0< r \leq 1/2$. Moreover, the  inequality (iii) in Theorem \ref{db1} holds, i.e. there exist $r_0$ and $C$ depending on $\theta$ in \eqref{theta} such that for all $r\leq r_0$ one has
\begin{equation}\label{bnj1}
\int_{\mathbb Q_r^+} |U|^2 y^a dX dt \leq C \int_{\mathbb Q_{r/2}^+} |U|^2 y^a dXdt.
\end{equation}
From this doubling estimate we can derive in a standard fashion the following inequality for all $r \leq \frac{r_0}{2}$
\begin{equation*}
\int_{\mathbb Q_r^+} |U|^2 y^a dX dt  \geq \frac{r^{L}}{C}   \int_{\mathbb Q_{r_0}^+} |U|^2 y^a dX dt,
\end{equation*}
where $L= \operatorname{log}_2 C$. Letting $c_0=\frac{1}{C}\int_{\mathbb Q_{r_0}^+} |U|^2 y^a dX dt$, and noting that $c_0>0$ in view of \eqref{nzero}, we can rewrite the latter inequality as
\begin{equation}\label{fvan1}
\int_{\mathbb Q_r^+} |U|^2 y^a dX dt  \geq c_0 r^{L}.
\end{equation}
Let now $r_j\searrow 0$ be a sequence such that $r_j \leq r_0$ for every $j\in \mathbb N$, and define
\[
U_j(X,t) = \frac{U(r_jX, r_j^2 t)}{\bigg(\frac{1}{r_j^{n+3+a}}\int_{\mathbb Q_{r_j}} |U|^2 y^a dX dt\bigg)^{1/2}}. 
\]
Note that on account of  \eqref{nzero} the functions $U_j$'s are well defined. Furthermore, a change of variable gives for every $j\in \mathbb N$
\begin{equation}\label{bound}
\int_{\mathbb Q_1^+} |U_j|^2 y^a dX dt=1.\end{equation}
Moreover from \eqref{bnj1} using change of variables we have\begin{equation}\label{nondeg}
\int_{\mathbb Q_{1/2}^+} |U_j|^2 y^a dX dt \geq C^{-1}.
\end{equation}
Moreover $U_j$ solves the following problem in $\mathbb Q_1^+$
\begin{equation}\label{expb1}
\begin{cases}
\Div( y^a \nabla U_j) + y^a \partial_t U_j= y^a r_j B \nabla U_j,
\\
\py U_j ((x,0),t) = r_j^{1-a} \tilde V(r_jx, r_j^2 t) U_j ((x,0),t).
\end{cases}
\end{equation}
From \eqref{bound} and the regularity estimates in Lemma \ref{regular1} and Lemma \ref{holdery} we infer that, possibly passing to a subsequence which we continue to indicate with $U_j$, we have $U_j \to U_0$ in $H^{\alpha}(\mathbb Q_{3/4}^+)$ up to $\{y=0\}$ and also $\py U_j ((x,0),t) \to \py U_0 ((x,0),t)$ uniformly in $\mathbb Q_{3/4}^+ \cap \{y=0\}$. Consequently, from \eqref{expb1} we infer that the blowup limit $U_0$ solves in $\mathbb Q_{3/4}^+$ 
\begin{equation}\label{expb2}
\begin{cases}
\Div( y^a \nabla U_0) + y^a \partial_t U_0=0,
\\
\py U_0 ((x,0), t) = 0.
\end{cases}
\end{equation}
We now observe that a change of variable and \eqref{fvan1} give
\[
\int_{B_1 \times [0,1)}  |U_j((x,0),t)|^2 dxdt \le c_0^{-1} r_j^{a+1-L} \int_{B_{r_j} \times [0, r_j^2)} |U((x,0),t)|^2 dx dt.
\]
Now using \eqref{vp1}  we infer that 
\[
\int_{B_1 \times [0,1)}  |U_j((x,0),t)|^2 dxdt  \to 0.
\]
Again, since $U_j \to U_0$ uniformly in $\mathbb Q_{1/2}^+$ up to $\{y=0\}$, we deduce that it must be $U_0 \equiv 0$ in $\mathbb Q_{1/2}^+ \cap \{y=0\}$. Moreover, since $U_0$ solves the problem \eqref{expb2}, we can now apply the weak unique continuation result in Proposition \ref{wucp}  to infer that  $U_0 \equiv 0$ in $\mathbb Q_{1/2}^+$. On the other hand, from the uniform convergence  of $U_j$'s in $\mathbb Q_{1/2}^+$ and  the non-degeneracy estimate \eqref{nondeg} we also have
\begin{equation}\label{nondeg1}
\int_{\mathbb Q_{1/2}^+} |U_0|^2 y^a dX dt \geq C^{-1},
\end{equation}
and thus $U_0\not\equiv 0$ in $\mathbb Q_{1/2}^+$. This contradiction leads to the conclusion that \eqref{crucial} must be true.  We now note that, away from the thin set $\{y=0\}$, $U$ solves a uniformly parabolic PDE with smooth coefficients and vanishes identically in the half-ball $\mathbb B_1^+$.  We can thus appeal to \cite[Theorem 1]{AV} to assert that $U$ vanishes to infinite order both in space and time in the sense of \eqref{vp} at every $(X,0)$ for $X \in \mathbb B_1$. At this point, we can  use the strong unique continuation  result  in   \cite[Theorem 1]{EF}  to finally conclude that $U(X,0) \equiv 0$ for $X\in \R^{n+1}_+$. The conclusion of the theorem thus follows. It is to be noted that although the  results in \cite{AV, EF} are stated for scalar equations, nevertheless the results are valid for systems which have decoupled principal part.

\end{proof}

\end{document}